\documentclass[11pt]{article}

\usepackage{amsmath, amssymb, mathtools}
\usepackage{amsthm}
\usepackage{mathrsfs} 
\usepackage{bbm}      
\usepackage{bm}	      
\usepackage{multirow, enumerate, graphicx, xcolor}
\usepackage{hyperref, comment, epstopdf, relsize}
\usepackage{subcaption}
\usepackage[labelfont=bf,margin=1cm]{caption}
\usepackage{lipsum}
\usepackage{soul}
\usepackage[noadjust,nocompress]{cite}

\hypersetup{
    colorlinks,
    linkcolor = {blue!50!black},
    citecolor = {blue!50!black},
    urlcolor  = {blue!80!black}
}

\numberwithin{equation}{section}

\theoremstyle{plain}
\newtheorem{theorem}{Theorem}[section] 
\newtheorem{lemma}[theorem]{Lemma}     
\newtheorem{proposition}[theorem]{Proposition}

\newtheorem{assumption}[theorem]{Assumption}
\theoremstyle{remark}

\DeclareMathAlphabet\mathbcal{OMS}{cmsy}{b}{n}
\newcommand\R{\mathbb{R}}

\newcommand{\bB}[1]{\boldsymbol{#1}}
\newcommand{\ttt}{\texttt{t}}
\renewcommand{\div}{{\rm div}}
\newcommand{\bdiv}{\mathbf{div}}
\newcommand\atopp[2]{\genfrac{}{}{0pt}{}{#1}{#2}}

\def\jump#1{\left[\!\left[#1\right]\!\right]}
\def\avg#1{\left\{\!\!\left\{#1\right\}\!\!\right\}}

\newcommand{\nsp}{k_{\rm b}}
\newcommand{\ssp}{k_{\rm f}}
\newcommand{\RR}{\mathbb{R}}

\renewcommand\L{{\mathrm L}}
\renewcommand\H{{\mathrm H}}
\newcommand{\RCtot}{\widetilde{\mathcal{R}}}
\newcommand{\bbeta}{\boldsymbol{\beta}}
\newcommand{\nn}{\boldsymbol{n}}
\newcommand{\biofilmindex}{\mathrm b}
\newcommand{\flowingindex}{\mathrm f}
\newcommand\kC{{k_{\biofilmindex}}}
\newcommand\kL{{k_{\flowingindex}}}
\newcommand{\rhoC}{\rho_{\biofilmindex}}
\newcommand{\rhoL}{\rho_{\flowingindex}}
\newcommand{\phiC}{{\phi_{\biofilmindex}}}
\newcommand{\phiL}{{\phi_{\flowingindex}}}
\newcommand{\vC}{{\bm{v}_{\biofilmindex}}}
\newcommand{\vL}{{\bm{v}_{\flowingindex}}}
\newcommand{\RC}{\mathcal{R}_{\bB{c}}}
\newcommand{\RL}{\mathcal{R}_{\bB{s}}}
\newcommand{\vrel}{\bm{v}_{\rm rel}}
\newcommand{\ppt}[1]{\frac{\partial #1}{\partial t}}
\newcommand{\abs}[1]{\left\lvert#1\right\rvert}
\newcommand{\aop}{\mathcal{F}}
\newcommand{\aconv}{a_h}
\newcommand{\adiff}{b_h}
\newcommand{\adiffc}{b^{\bB{c}}_h}
\newcommand{\adiffs}{b^{\bB{s}}_h}
\newcommand{\spaceq}{\mathbcal{P}_{2,0}^{\rm cont}(\Omega)}
\newcommand{\spacep}{\mathcal{P}_{0,0}(\Omega)}
\newcommand{\Mup}{M^{\uparrow}}
\newcommand{\Mdn}{M^{\downarrow}}
\newcommand{\MupC}{M^{\uparrow}_{\bB{c}}}
\newcommand{\MdnC}{M^{\downarrow}_{\bB{c}}}
\newcommand{\upwF}{\Phi_{e}}
\newcommand{\sF}{\widetilde{\Phi}^{\bB{s}}_{e}}
\newcommand{\cF}{\widetilde{\Phi}^{\bB{c}}_{e}}
\newcommand{\uF}{\widetilde{\Phi}_{e}}

\newcommand{\boundr}{{\rm R}}
\newcommand{\Ltwo}{0,\Omega}

\usepackage{siunitx}
\DeclareSIUnit\ourtime\second
\DeclareSIUnit\concentration{\kilogram\per\metre\cubed}
\DeclareSIUnit\acceleration{\metre\per\ourtime\squared}
\DeclareSIUnit\mobility{\metre\squared\per\ourtime}
\DeclareSIUnit\capillary{\metre}
\DeclareSIUnit{\kappaunit}{\metre\squared}
\DeclareSIUnit\viscosity{\pascal\ourtime}

\begin{document}
\title{
	\sc An invariant-region-preserving scheme for a convection-reaction-Cahn--Hilliard multiphase model of biofilm growth in slow sand filters
}
\author{
	{\sc Julio Careaga}\thanks{Departamento de Matem{\'a}tica, Universidad del B{\'\i}o B{\'\i}o, Chile, email:{\tt jcareaga@ubiobio.cl}.}
	\quad
	{\sc Stefan Diehl}\thanks{Centre for Mathematical Sciences, Lund University, Sweden, email:{\tt stefan.diehl@math.lth.se}.}
	\quad
	{\sc Jaime Manr{\'\i}quez}\thanks{Centre for Mathematical Sciences, Lund University, Sweden, email:{\tt jaime.manriquez@math.lth.se}.}
}
\date{}

\maketitle

\begin{abstract}
\noindent 
A multidimensional model of biofilm growth present in the supernatant water of a Slow Sand Filter is derived.
The multiphase model, consisting of solid and liquid phases, is written as a convection-reaction system with a Cahn--Hilliard-type equation with degenerate mobility coupled to a Stokes-flow equation for the mixture velocity. 
An upwind discontinuous Galerkin approach is used to approximate the convection-reaction equations, whereas an $H^1$-conforming primal formulation is proposed for the Stokes system.
By means of a splitting procedure due to the reaction terms, an invariant-region principle is shown for the concentration unknowns, namely non-negativity for all phases and an upper bound for the total concentration of the solid phases.
Numerical examples with reduced biofilm reactions are presented to illustrate the performance of the model and numerical scheme.
\end{abstract}
\smallskip

\smallskip


\section{Introduction}\label{sec:intro}
\subsection{Scope}\label{sec:intro:scope}
Slow Sand Filters (SSFs) are a widely used technology in water purification processes \cite{Maiyo2023}, which offer advantages across both operational and environmental aspects, enhancing their overall value and impact. 
The principle behind an SSF is that filtration of water is carried out through a bed of fine sand at low velocity, retaining the organic and inorganic suspended matter in the upper layer of the sand bed in the form of growing biofilm, a fluffy sponge-like structure commonly referred to as the \emph{Schmutzdecke}, whose main purpose is to remove pathogens and other harmful bacteria as they are attached to the biofilm and decomposed. 
Despite the simplicity of its design and entire operation, from a modelling point of view, an SSF represents an intricate blend of physical, chemical and biological processes. 
In order to develop an appropriate mathematical model that captures at least part of the complexity of an SSF, one must take into account a variety of concentrations of biological matter, an assortment of reactions and other types of interactions and the fluid dynamics of the mixture. 

The one-dimensional multiphase model derived by Diehl et al.~\cite{Diehl2025} is based on a coupled system of partial differential equations (PDEs) for the mass balance of the component concentrations of the biofilm and flowing suspension, in which the interactions between species are modelled by nonlinear source terms. 
The main behaviour of the biofilm concentration is described by a Cahn--Hilliard-type equation, which is a nonlinear PDE with fourth-order spatial derivatives.
Such an equation models the behaviour of a viscous fluid (in this case the biofilm) where the cohesion energy modelled via a potential is relevant.
In a natural extension to several spatial dimensions, the unknown volume-averaged velocity is no longer uniquely determined by an incompressibility condition; hence, an additional momentum balance for the mixture is needed.

The first purpose of this contribution is to extend a part of the model in~\cite{Diehl2025} to two spatial dimensions, namely the biofilm growth from the packed sand bed upwards into the supernatant water.
Multiple concentrations of biofilm components (solid phases) and flowing-suspension solutes (liquid phases) are introduced, each one modelled by mass-balance convection-reaction equations depending on the total concentration of biofilm, with all components coupled via nonlinear reaction terms. 
Stokes-flow equations, modelling the volume-averaged velocity of the mixture, are coupled to the rest of the unknowns through the viscosity and forcing terms due to gravity and Korteweg capillary forces.
We obtain thus a system of convection-reaction and Cahn--Hilliard type for the concentrations coupled to Stokes-flow equations for the velocity field and pressure.

The second purpose of this article is to present a reliable invariant-region-preserving numerical scheme, which also preserves the incompressibility of the mixture flow.
The numerical scheme is based on the one by Acosta-Soba et al.~\cite{Acosta2023b} for their model with one biofilm and one liquid phase.
Our model has additional subphases.
The advective Cahn--Hilliard equation with degenerate mobility poses several difficulties in terms of numerical approximation, because of the high-order derivatives and the nonlinear coefficient functions in the higher-order terms~\cite{Abels2013, Guillen2014, Guillen2024a}.
The upwind discontinuous Galerkin approach in~\cite{Acosta2023b} ensures both proper control of the degenerate mobility terms and positivity for the case of piecewise constant approximations in the absence of reaction terms. 
To be able to prove an invariant-region-preserving property, we propose a simple time splitting where one time step is taken without reactions and one with.

\subsection{Related work}\label{sec:intro:related-work}
Biofilm growth, which is an essential component in slow sand filtration, is described by Cahn--Hilliard equations with degenerate mobility in~\cite{Klapper2006,Zhang2008}.
In~\cite{Diehl2025}, a multiphase model of the biofilm growth in the entire SSF was introduced, where the biofilm and soluble substrates are coupled via convection-reaction equations. 
Among the works dedicated to the modelling of different aspects of SSFs based on PDEs, we refer to \cite{Schijven2013, Song2020, Trikannad2023}. 
These models are mainly formulated in one spatial dimension, with the exception of \cite{Song2020}, which has a particulate-based model. 
In \cite{Schijven2013} and \cite{Trikannad2023}, the focus is on the removal of pathogens, which is modelled by by linear convection-diffusion-reaction equations.
Cahn--Hilliard equations were originated in a series of works~\cite{Cahn1959,Cahn1958,Cahn1959b}, and we refer to \cite{Jingxue1992,Elliott1996} for well-posedness analysis and~\cite{Wu2022} for a review on mathematical aspects and analytical issues of these equations. 

There is a number of numerical methods developed to approximate Cahn--Hilliard equations in two-dimensional domains: from conservative finite difference schemes \cite{Furihata2001, Li2013, Shin2011}, including the specific one developed in~\cite{Diehl2025}, to finite element \cite{Banas2008, Barrett1999} and discontinuous-Galerkin approaches \cite{Acosta2023b,Kay2009,Xia2007}. 
We give special attention to the structure-preserving upwind discontinuous-Galerkin scheme in~\cite{Acosta2023b} for the analysis here.
For the stationary Stokes system, there exist different types of formulations depending on the boundary conditions, regularity of the velocity field and conservation properties. 
Some examples are pseudo-stress-based~\cite{Gatica2010}, augmented~\cite{Gatica2011} and primal formulations~\cite{Boffi2013}. 

A system containing multicomponent convection-reaction equations can be found in the so-called reactive-sedimentation model \cite{Burger2021,Burger2023b,Careaga2024}, where the velocities of the solid and liquid components depend on the total concentration of solids, and the coupling is due to nonlinear biochemical reaction terms. 
Both the splitting technique and the strategies to show the invariant-region-preserving properties in the presence of reaction terms are based on~\cite{Burger2023b}.
Another multicomponent Cahn--Hilliard-based model can be found in~\cite{Clavijo2019} with the difference that each component is modelled by a Cahn--Hilliard equation, whereas our model has one Cahn--Hilliard-type equation for the total biofilm concentration.

There are different proposals for the coupling of the Cahn--Hilliard equation to the Navier--Stokes equations.
We refer to \cite[Table~1]{Eikelder2024} for a comparison of various models. 
Some analyses account for the relative mass flux between species~\cite{Abels2013, Stoter2023, Eikelder2024} in the momentum equation, whereas other models do not incorporate it~\cite{Giorgini2020} or consider its effects in the mass balance instead~\cite{Guo2022}. 
The incorporation of capillary forces also varies between models, with some stating it in terms of the gradient of the chemical potential~\cite{Guo2022, Han2015} or some expression including the gradient of the order parameter \cite{Abels2013, Dlotko2022, Han2021}.
We point to~\cite[Remark~2.1]{Eikelder2024} for an identity relating these different expressions, although we point out that the discrete spaces used for the numerical scheme plays a role in our choice of capillary-force term.

Other applications of Cahn--Hilliard equations in biology include membrane separation~\cite{Elson2010} and wound healing~\cite{Chatelain2011a, Chatelain2011b, Cherfils2014, Khain2008}.
Regarding the coupled Cahn--Hilliard--Navier--Stokes equations when applied to tumour growth, we mention \cite{Acosta2025,Zou2022,Guillen2024b,Guillen2024a}.

\subsection{Outline of the paper}\label{sec:intro:outline}
The rest of the paper is organized as follows. In Section \ref{sec:model}, we provide a derivation of our model leading to the governing equations. Constitutive assumptions, boundary conditions and prescribed model functions are also described in this section. Section~\ref{sec:scheme} is dedicated to the description and development of the numerical scheme employed. In Section~\ref{sec:scheme:prelim}, we introduce the mesh, the polynomial spaces and some operators used in our scheme. The fully discrete formulation of our model equations is presented in Section~\ref{sec:scheme:discontinuous}. The properties of the developed numerical scheme, including the main invariant-region result, are given in Section~\ref{sec:scheme:properties}. Finally, illustrative numerical simulations are shown in Section~\ref{sec:simulations}, and concluding remarks are presented in Section~\ref{sec:conclusions}.

\section{Model equations}\label{sec:model}

\subsection{Phases and mass balances}
We consider a biofilm matrix immersed in a solution comprised of water and substrates. The physical system is modelled as a mixture of solid and liquid components within a bounded domain $\Omega\subset \mathbb{R}^{d}$, $d\in \{2,3\}$, such that its boundary $\Gamma :=\partial\Omega$ is polyhedral, with outward unit normal vector $\bB{n}$. The polyhedral condition of the domain is only imposed to properly address the numerical scheme in Section~\ref{sec:scheme}.
The system consists of $\kC$ solid phase components (biofilm) with concentrations $c^{(1)}, \ldots, c^{(\kC)}$ and $\kL$ liquid phase components with concentrations $s^{(1)}, \ldots, s^{(\kL)}$, which are functions that depend on the space coordinate $\bB{x}\in \Omega$ and time $t\geq 0$. 
The concentrations of solid species are grouped in the vector $\bB{c} \coloneqq (c^{(1)}, \ldots, c^{(\kC)})^{\ttt}\in\mathbb{R}^{\kC}$, and the concentrations of substrates in $\bB{s}\coloneqq(s^{(1)}, \ldots, s^{(\kL)})^{\ttt}\in\mathbb{R}^{\kL}$, where $(\cdot)^{\ttt}$ is the transpose operation. We let the total concentrations of solids and liquids be, respectively,
\begin{align*}
	u(\bB{x},t) \coloneqq \sum_{i = 1}^{\kC} c^{(i)}(\bB{x},t), \qquad s_{\rm tot}(\bB{x},t) \coloneqq \sum_{i = 1}^{\kL} s^{(i)}(\bB{x},t),\qquad \forall \bB{x}\in \Omega,\quad t\geq 0.
\end{align*}
The densities of all solid components are assumed to be equal to $\rhoC > 0$, and similarly for the liquid components we assume all densities are equal to $\rhoL < \rhoC$. 
The biofilm density~$\rhoC$ is however assumed to be only slightly greater than~$\rhoL$.
To properly derive our model equations, we define the volume fractions related to the solid and liquid phases by
\begin{equation}\label{eq:sum2one}
\phiC \coloneqq\frac{u}{\rhoC},\qquad
\phiL \coloneqq\frac{s_{\rm tot}}{\rhoL},\qquad
\phiC + \phiL = 1.
\end{equation}
The mass balances for all solid and liquid components are
\begin{align}
	\ppt{c^{(i)}} + \div\big(\bm{\vC} c^{(i)}\big) = \RC^{(i)}(\bB{c},\bB{s}), \quad i = 1, \ldots, \kC, \label{eq:mass:c} \\
	\ppt{s^{(j)}} + \div\big(\bm{\vL} s^{(j)}\big) = \RL^{(j)}(\bB{c},\bB{s}), \quad j = 1, \ldots, \kL, \label{eq:mass:s}
\end{align}
for all $\bB{x}\in \Omega$ and $t\in(0,T]$, where $\vC\in \mathbb{R}^{d}$ and $\vL\in \mathbb{R}^{d}$ are the velocities of the solid and liquid phases, respectively, and $\mathrm{div}(\cdot)$ is the divergence operator in $\mathbb{R}^d$.
The functions $\RC^{(i)}$ and $\RL^{(j)}$ for all $i=1,\ldots,\kC$ and $j=1,\ldots,\kL$ are the reactive terms corresponding to the solid and liquid components, respectively, typically nonlinear functions involving monod factors.  
With the assumption that the total mass is preserved during reactions, we have
\begin{equation}\label{eq:sumR}
	\sum_{i=1}^{\kC}\RC^{(i)}+\sum_{i=1}^{\kL}\RL^{(i)}=0.
\end{equation}

\subsection{Velocities and momentum balance for the mixture}
We introduce the density of the mixture 
\begin{equation}\label{eq:densmix}
	\rho(\phiC) \coloneqq \rhoC\phiC + \rhoL\phiL = \phiL + {\Delta\rho}\, \phiC,
\end{equation}
where $\Delta\rho\coloneqq\rhoC-\rhoL$, and then the mass-averaged bulk velocity~$\bB{v}$, volume-averaged bulk velocity~$\bB{q}$ and relative velocity between phases~$\vrel$ by
\begin{align*}
	\bB{v} \coloneqq \frac{\phiC\rhoC\vC + \phiL\rhoL\vL}{\rho(\phiC)},\qquad \bB{q} \coloneqq \phiC\vC + \phiL\vL, \qquad \vrel \coloneqq \vC - \vL.
\end{align*}
With the excess pore pressure $p$ and a constitutive function $\nu\coloneqq\nu(u)$ for the viscosity of the mixture, we define the Cauchy stress tensor for the mixture
\begin{equation*}
\mathbb{S}(u,\bB{v},p) \coloneqq \nu(u)\bB{\varepsilon}(\bB{v}) - p\mathbb{I},\quad\text{where}\quad \bm\varepsilon(\bm{v}) \coloneqq \frac{1}{2}(\nabla \bm{v} + (\nabla \bm{v})^{\ttt})
\end{equation*}
is the symmetric part of the gradient or rate-of-strain tensor and $\mathbb{I}$ is the identity matrix of $\mathbb{R}^{d\times d}$.
The momentum balance for the mixture is given by
\begin{equation}\label{eq:mombal}
	{\rm D}^{\bB{v}}_t \big({\rho(\phiC)}\bB{v}\big)  = \bdiv\big(\mathbb{S}(u,\bB{v},p)\big) + \mathbf{f},
\end{equation}
where ${\rm D}^{\bB{v}}_t(\bB{w}) \coloneqq \partial_t \bB{w} + \bdiv (\bB{w} \otimes \bB{v})$ is the material derivative and $\mathbf{f}$ contains additional forces per unit volume acting on the mixture.
The operator $\bdiv(\cdot)$ is defined for tensors in $\mathbb{R}^{d\times d}$ and corresponds to the divergence operator acting along the rows of the tensor. 

Now, we observe that the relation between $\bB{v}$ and $\bB{q}$ can be derived by first expressing the phase velocities in terms of~$\bB{q}$ and~$\vrel$ as
\begin{align}\label{eq:vCvL}
	\vC = \bB{q} + \phiL\vrel = \bB{q} + (1-\phiC)\vrel , \qquad  \vL = \bB{q} - \phiC\vrel.
\end{align}
Then we substitute these expression into the definition of the mass-averaged velocity~\eqref{eq:densmix} to obtain
\begin{align*}
	\rho(\phiC)\bB{v} &= \phiC\rhoC\vC + \phiL\rhoL\vL 
	= \rho(\phiC)\bB{q} + {\Delta\rho}\,\phiC(1-\phiC)\vrel,
\end{align*}
from which we obtain the identity
\begin{equation*}
	\bB{v} = \bB{q} + \frac{{\Delta\rho}}{\rho(\phiC)}\phiC(1-\phiC)\vrel.
\end{equation*}
Under the assumption that ${{\Delta\rho}}/{\rho(\phiC)}\approx 0$ and that the relative velocity~$\vrel$ has the same order of magnitude as~$\bB{q}$, the above identity implies the approximation $\bB{v} \approx \bB{q}$ and allows us to state the momentum balance in terms of $\bB{q}$ rather than $\bB{v}$.
Furthermore, dividing~\eqref{eq:mass:c} by~$\rhoC$ and~\eqref{eq:mass:s} by~$\rhoL$ and adding the resulting equations, one gets
\begin{equation*}
	\div(\bB{q}) = \frac{1}{\rhoC}\sum_{i=1}^{\kC}\RC^{(i)}+\frac{1}{\rhoL}\sum_{i=1}^{\kL}\RL^{(i)}
	= \frac{1}{\rhoC\rhoL}\left( \rhoL\left( \sum_{i=1}^{\kC}\RC^{(i)} + \sum_{i=1}^{\kL}\RL^{(i)} \right) + {\Delta\rho} \sum_{i=1}^{\kC}\RL^{(i)} \right).
\end{equation*}
Then, thanks to Equation~\eqref{eq:sumR} and {${\Delta\rho}/(\rhoC\rhoL)\approx0$ (which is in agreement with the previous assumption), the right-hand side is small and we approximate it by zero, so that the volume-averaged velocity field is divergence free. Moreover, for the modelling of SSFs, we assume that the flow regime is such that the inertial terms in the momentum balance~\eqref{eq:mombal} are much smaller than the viscous terms, wherefore we neglect the material derivative ${\rm D}^{\bB{v}}_t(\rho(\phiC)\bB{v})$.
Consequently, the momentum and mass balances for the mixture in terms of $\bm{q}$ are given by the Stokes-flow equations
\begin{equation*}
	\begin{aligned}
		-\bdiv\big(\nu(u)\bB{\varepsilon}(\bB{q}) - p\mathbb{I}\big) &= \mathbf{f}, \\
		\div (\bm{q})&=0,
	\end{aligned}
\end{equation*}
where an explicit form for $\mathbf{f}$ is stated below. 

\subsection{Constitutive assumptions}
In reactive sedimentation~\cite{Burger2021}, which is also influenced by gravity, the relative velocity is expressed by a decreasing function of the solids concentration due to the increased hindrance of sedimentation with increased concentration.
In the case of biomass moving in a liquid, we make the following constitutive assumption for the relative velocity:
\begin{equation}\label{eq:vrel}
	\vrel \coloneqq -\lambda(1-\phiC)^\gamma\nabla\mu,
\end{equation}
where the unknown variable $\mu:=\mu(\bB{x},t)$, $\bB{x}\in\Omega$, $t\in(0,T]$, is the chemical potential of the mixture (here, dimensionless), $\gamma\geq 0$ is an exponent related to the squeezing of biofilm particles~\cite{Wang2007} and $\lambda>0$ is a constant mobility parameter [\unit{\mobility}].
In essence, we are proposing that, in the absence of bulk flow, the mixture moves proportionally to a potential gradient, and finds a steady state once the chemical potential is spacially homogeneous. Following \cite{Klapper2006}, we assume that the potential $\mu$ satisfies
\begin{equation} \label{eq:mu}
	\mu = -\kappa\,\div(\nabla\phiC) + \Psi'_0(\phiC)\qquad \text{with}\qquad
	\Psi_0(\phiC) \coloneqq \phi_{\rm b}^3\left(\phiC-\frac{4}{3}\phi_{\ast}\right),
\end{equation}
where $\kappa$ is a large gradient-penalizing constant [\unit{\kappaunit}], $\Psi_0$ is the free energy density of the mixture, and $\phi_{\ast}\in(0,1)$ is the stable biofilm volume fraction if no other forces are present. 
By defining $\Psi(u)\coloneqq\rhoC\Psi_0(u/\rhoC)$, we can write the right-hand side of~\eqref{eq:mu} in terms of the total concentration~$u$ instead
\begin{equation*}
	\mu = -\frac{\kappa}{\rhoC}\,\div(\nabla u) + \Psi'(u),
\end{equation*}
and we assume that the following bound exists:
\begin{equation}\label{eq:bound:dPsi}
	\Psi'_{\rm max} \coloneqq \max\big\{|\Psi'(u)|: 0\leq u \leq 1\big\}.
\end{equation}
We assume that the mixture viscosity is the convex combination
\begin{align*}
 \nu(u) \coloneqq \nu_{\rm b}\Big(1 - \dfrac{u}{\rhoC}\Big) + \nu_{\rm f}\dfrac{u}{\rhoC}
\end{align*}
of the reference viscosities $\nu_{\rm b}$ and $\nu_{\rm f}$.

The body force per unit volume $\mathbf{f}$ can be written as
\begin{equation}
	\mathbf{f}(u,\mu,\nabla u)\coloneqq -\dfrac{g{\Delta\rho}}{\rhoC}u\bB{k}+{\eta}\mathbf{f}_{\rm kw}(\mu,\nabla u), \label{eq:source:stokes}
\end{equation}
where $g$ is the acceleration of gravity, $\bB{k}$ is the unit normal vector pointing in the direction opposite to gravity, $\eta$ [\unit{\capillary}] is a positive constant related to the surface tension and $\mathbf{f}_{\rm kw}$ is a force per unit volume due to elastic energy.
We will use the following function for the modelling of capillary forces as in \cite{Dlotko2022, Han2021}:
\begin{align*}
	\mathbf{f}_{\rm kw}(\mu,\nabla u) \coloneqq \mu\nabla u = \Big(-\frac{\kappa}{\rhoC}\,\div(\nabla u) + \Psi'(u)\Big)\nabla u \eqqcolon \widetilde{\mathbf{f}}_{\rm kw}(u,\div(\nabla u),\nabla u),
\end{align*}
where the alternative expression $\widetilde{\mathbf{f}}_{\rm kw}$ will be used for a simulation in Section~\ref{sec:scheme}.

\subsection{Governing equations}
Into the mass balances~\eqref{eq:mass:c} and \eqref{eq:mass:s}, we substitute the expressions~\eqref{eq:vCvL} for the phase velocities~$\vC$ and~$\vL$ in terms of~$\bB{q}$ and the relative velocity~$\vrel$, where the latter has the constitutive assumption~\eqref{eq:vrel}.
Then the mass-balances become
\begin{align}
	\ppt{c^{(i)}} + \div\Big(c^{(i)}\big(\bB{q} - \lambda (1-u/\rhoC)^{1+\gamma}\nabla\mu\big) \Big) = \RC^{(i)}, \quad i = 1, \ldots, \kC, \label{eq:mass:c:new} \\
	\ppt{s^{(j)}} + \div\Big(s^{(j)}\big(\bB{q} + \lambda (u/\rhoC)(1-u/\rhoC)^{\gamma}\nabla\mu\big) \Big) = \RL^{(j)}, \quad j = 1, \ldots, \kL. \label{eq:mass:s:new}
\end{align}
Summing all the equations of \eqref{eq:mass:c:new}, we obtain the following equation for the total solids concentration:
\begin{equation}\label{eq:CH}
	\ppt u + \div\big(u\bB{q} - \lambda u(1-u/\rhoC)^{1+\gamma}\nabla\mu\big) = \RCtot(\bB{c},\bB{s}),
\end{equation}
where $\RCtot(\bB{c},\bB{s}) := \RC^{(1)}(\bB{c},\bB{s})+\ldots+\RC^{(\kC)}(\bB{c},\bB{s})$ is the reaction term corresponding to the total solids concentration. 
Equation~\eqref{eq:CH} is a Cahn--Hilliard-type equation with advection and degenerate-mobility function 
\begin{equation}\label{eq:M}
	M(u) \coloneqq \begin{cases}
		\lambda u(1-u/\rhoC)^{1+\gamma}, & 0 \leq u \leq \rhoC, \\
		0, & \text{otherwise},
	\end{cases}
\end{equation}
where we have extended the mobility by zero to values outside of the admissible-values region.
We note that by setting $\gamma = 0$, we recover the usual Cahn--Hilliard mobility used in models such as the one in~\cite{Acosta2023b}, whereas $\gamma=1$ corresponds to the model studied in~\cite{Chatelain2011b}.

To summarize the model, we introduce the functions
\begin{align*}
	M_{\bB{c}}(u) &\coloneqq \begin{cases}
		\lambda(1-u/\rhoC)^{1+\gamma},\phantom{(\rhoC)} & 0 \leq u \leq \rhoC, \\
		0, &\text{otherwise},
	\end{cases}\\
	M_{\bB{s}}(u) &\coloneqq \begin{cases}
	\lambda (u/\rhoC)(1-u/\rhoC)^{\gamma}, & 0 \leq u \leq \rhoC, \\
	0, &\text{otherwise},
\end{cases}
\end{align*}
which satisfy $uM_{\bB{c}}(u)=M(u)$ and $(\rhoC-u)M_{\bB{s}}(u)=M(u)$ for $u\in\R$.

The governing equations are the following for the unknowns $(\bB{c},\bB{s},u,\mu)\in \RR^{\kC+\kL+2}$ and $(\bB{q},p)\in \RR^{d+1}$, which are defined for all $\bB{x}\in \Omega$ and $t\in(0,T\,]$:
\begin{subequations} \label{syst:main}
	\begin{align}
		\dfrac{\partial c^{(i)}}{\partial t}  +\div\Big(c^{(i)}\big(\bB{q} - M_{\bB{c}}(u)\nabla\mu\big)\Big) & = \mathcal{R}_{\bB{c}}^{(i)}(\bB{c},\bB{s}),\qquad i = 1,\dots,\nsp,\label{eq:solids}\\[1ex]
		\dfrac{\partial s^{(j)}}{\partial t}  + \div\Big(s^{(j)}\big(\bB{q} + M_{\bB{s}}(u)\nabla\mu \big)\Big) & = \mathcal{R}_{\bB{s}}^{(j)}(\bB{c},\bB{s}),\qquad j = 1,\dots,\ssp, \label{eq:substrates}\\[1ex]
		\dfrac{\partial u}{\partial t} +\div\big(u\bB{q} - M(u)\nabla\mu\big) &= \mathcal{\widetilde{R}}(\bB{c},\bB{s}),\label{eq:total}\\
		\mu + \dfrac{\kappa}{\rhoC} \nabla^2 u - \Psi'(u) &= 0,\label{eq:potencial} \\[1ex]
		-\bdiv\big(\nu(u)\bB{\varepsilon}(\bB{q}) - p\mathbb{I}\big)  &= {-\frac{g\Delta\rho}{\rhoC} u\bm{k} + \eta \mu\nabla u},\label{eq:stokesmom}\\[1ex]
		\div(\bB{q}) &= 0. \label{eq:stokesdiv}
	\end{align}
\end{subequations}
Since the sum of the $\bB{c}$ components is equal to $u$, one of the equations of~\eqref{eq:solids} and~\eqref{eq:total} is redundant and can be obtained from the other.
In Section~\ref{sec:scheme}, we show that it is sufficient to define the numerical fluxes properly to preserve this property at each time iteration. 
Furthermore, assuming that the water concentration is~$s^{(k_{\rm f})}$ and that this do not influence any reaction, Equation~\eqref{eq:substrates} for $j=k_{\rm f}$ can be obtain after all the other equations have been solved because of~\eqref{eq:sum2one}.
System~\eqref{syst:main} is supplemented with the following conditions:
\begin{align}
 	\bB{q}&=\bB{0}\quad\text{on}\quad \partial\Omega, \label{def:boundary:conditions:q}\\
		M(u)\nabla\mu\cdot\nn&= 0 \quad\text{on}\quad \partial\Omega, \label{def:boundary:conditions:mu}
	\\ \int_{\Omega}p &= 0.
\end{align}
Although our results only cover the case of zero boundary conditions for~$\bB{q}$, we will in Section~\ref{sec:simulations} show one simulation with non-zero Dirichlet condition for~$\bB{q}$ at in- and outlets.

\subsection{Reaction terms and stoichiometric matrices}\label{sec:model:reactions}
The reaction terms have the following form:
\begin{align*}
	\mathbcal{R}_{\bB{c}}(\bB{c},\bB{s}) = \bB{\sigma}_{\bB{c}}\bB{r}(\bB{c},\bB{s}),\qquad  \mathbcal{R}_{\bB{s}}(\bB{c},\bB{s}) = \bB{\sigma}_{\bB{s}}\bB{r}(\bB{c},\bB{s}),
\end{align*}
where $\bB{\sigma}_{\bB{c}}\in\mathbb{R}^{\nsp\times l}$ and $\bB{\sigma}_{\bB{s}}\in\mathbb{R}^{\ssp\times l}$ are constant and dimensionless stoichiometric matrices, and $\bB{r}(\bB{c},\bB{s})\in \mathbb{R}^l$, $l\in \mathbb{N}$, is the vector of reactions.
In order to prove positivity-preserving results in Lemma~\ref{thm:maxprinciple:ODEs} below and to ensure conservation of mass, we make the following assumptions and notation.
We assume that the reaction rates are bounded functions $\bB{r}(\bB{c},\bB{s})\geq \bB{0}$ for all $(\bB{c},\bB{s})\in \mathbb{R}^{\nsp}\times \mathbb{R}^{\ssp}$.
Given a component~$k$ of concentration $\bB{\xi}\in\{\bB{c},\bB{s}\}$, we define the index sets of negative and positive stoichiometric coefficients, respectively, by
\begin{align*}
	\mathcal{J}_{\bB{\xi},k}^-:=\Big\{j\in\mathbb{N}:\,\sigma_{\bB{\xi}}^{(k,j)}<0\Big\}, \qquad
	\mathcal{J}_{\bB{\xi},k}^+:=\Big\{j\in\mathbb{N}:\,\sigma_{\bB{\xi}}^{(k,j)}>0\Big\} \quad \text{for}\quad\bB{\xi}\in\{\bB{c},\bB{s}\}.
\end{align*}

\begin{assumption}
	Let $\bB{\xi}\in\{\bB{c},\bB{s}\}$. The stoichiometric matrices and vector of reaction rates satisfy the following hypotheses
	\begin{enumerate} \label{asumption:reactionterms}
		\item The vector of ones belongs to ${\rm ker}(\bB{\sigma}_{\bB{c}}^{\tt t}, \bB{\sigma}_{\bB{s}}^{\tt t})\,$. This is overall mass preservation and implies~\eqref{eq:sumR}.
		\item For each $j=1,\dots, l$, there holds $r^{(j)}\in \mathrm{C}\big(\bar{\mathbb{R}}_+^{\nsp},\bar{\mathbb{R}}_+^{\ssp}\big)$, with $r^{(j)}$ locally Lipschitz continuous.
		\item For every $j\in \mathcal{J}_{\bB{\xi},k}^-$ there exists a bounded function $\bar{r}^{(j)}\in  \mathrm{C}\big(\bar{\mathbb{R}}_+^{\nsp},\bar{\mathbb{R}}_+^{\ssp}\big)$ with $\bar{r}^{(j)}(\bB{c},\bB{s})>0$ if $\bB{c}>\bB{0}$ and $\bB{s} > \bB{0}$, and $\bar{r}^{(j)}(\bB{c},\bB{s})=0$ if $\bB{c} = \bB{0}$ and $\bB{s}=\bB{0}$, such that
		\begin{align*}
			r^{(j)}(\bB{c},\bB{s}) =  \bar{r}^{(j)}(\bB{c},\bB{s})\xi^{(k)}.
		\end{align*}
		(The reason for this assumption is natural; if a component $\xi^{(k)}$ is consumed down to zero concentration, then no more of that component can be consumed.)
		There exists an upper bound $\boundr>0$ such that $\bar{r}^{(j)}(\bB{c},\bB{s}) \leq \boundr$ for all $j\in \mathcal{J}_{\bB{\xi},k}^-$ and $1\leq k \leq \nsp+\ssp$ and all $(\bB{c},\bB{s})\in\mathbb{R}^{\nsp}\times \mathbb{R}^{\ssp}$ with $\bB{0}\leq\bB{c}\leq\rho_{\rm b}\bB{1}$ and $\bB{s}\geq\bB{0}$.
		\item There exists an $\boundr_{\max} > 0$ such that $|\mathcal{R}_{\bB{c}}^{(i)}(\bB{c},\bB{s})| \leq \boundr_{\max}$ and $|\mathcal{R}_{\bB{s}}^{(j)}(\bB{c},\bB{s})| \leq \boundr_{\max}$, for all $0\leq i\leq \nsp$, $0\leq j\leq \ssp$ and all $(\bB{c},\bB{s}) \in \mathbb{R}^{\nsp}\times \mathbb{R}^{\ssp}$ with $\bB{0}\leq\bB{c}\leq\rho_{\rm b}\bB{1}$ and $\bB{s}\geq\bB{0}$.
		\item There exists an $\varepsilon > 0$ such that
		\begin{align*}
			\mathbcal{R}_{\bB{c}}(\bB{c},\bB{s}) = \bB{0}\,\quad \text{for all}\quad u = \sum_{i=1}^{\nsp} c^{(i)}\, \geq\, \rhoC - \varepsilon.
		\end{align*}
	\end{enumerate}
\end{assumption}

\section{Numerical scheme} \label{sec:scheme}
\subsection{Preliminaries}\label{sec:scheme:prelim}
In this section we consider a bounded polygonal domain $\Omega\subset \RR^d$ with $d\in\{2,3\}$, and a regular triangulation of $\Omega$ denoted by $\mathcal{T}_h=\{K\}_{K\in \mathcal{T}_h}$ with meshsize $h\in \RR^{+}$, where the elements $K\in\mathcal{T}_h$ are triangles (if $d=2$) or a tetrahedra (if $d=3$). 
For the triangulation $\mathcal{T}_h$, we assume that there exists $\widetilde{C}>0$ such that
\begin{align}\label{eq:admissiblemesh}
	\widetilde{C}\, h_K^{d}\leq |K|\quad\text{and}\quad \abs{\partial K} \leq \widetilde{C}^{-1}h_K^{d-1},\qquad \forall K\in \mathcal{T}_h,
\end{align}
where $h_K$ and $|K|$, denote the diameter and measure of $K\in \mathcal{T}_h$, respectively.
Condition \eqref{eq:admissiblemesh} is to fulfil the requirements of an admissible mesh as described in \cite{Eymard2000}, which is going to be used later on in Lemma~\ref{lem:aux:existence}.
Furthermore, we define $\mathcal{E}_h$ as the set of edges of $\mathcal{T}_h$ (or faces if $d=3$), which is composed of interior and boundary edges, grouped into sets $\mathcal{E}_h^{\rm int}$ and $\mathcal{E}_h^{\partial}$, respectively, such that $\mathcal{E}_h=\mathcal{E}_h^{\rm int}\cup\mathcal{E}_h^{\partial}$. 
Then, given $K\in \mathcal{T}_h$ and neighbour element $L\in \mathcal{T}_h$, the unique shared edge $e\in\mathcal{E}_h^{\rm int}$ is such that $e=\partial K\cap\partial L$ with unit normal vector $\nn_e$ pointing to $L$. 
For boundary edges $e\in\mathcal{E}_h^{\partial}$, we assume the unit normal vector $\nn_e$ is pointing outwards from $\Omega$. 
We let $\mathcal{P}_\ell(K)$ be the local space of polynomial functions of degree less than or equal to $\ell\geq 0$ over $K\in \mathcal{T}_h$, for all $K\in \mathcal{T}_h$. 
In addition, we define the discontinuous and continuous global polynomial spaces by
\begin{align*}
	&\mathcal{P}_{\ell}(\Omega) \coloneqq\Big\{\varphi:\Omega\to \mathbb{R}:\quad \varphi|_{K}\in \mathcal{P}_{\ell}\qquad \forall K\in \mathcal{T}_h(\Omega)\Big\},\qquad \mathcal{P}^{\rm cont}_{\ell}(\Omega) \coloneqq \mathcal{P}_{\ell}(\Omega)\cap C(\bar{\Omega}).
\end{align*}
For each piecewise polynomial function $\varphi_h$, where the subscript $h$ corresponds to the mesh size of $\mathcal{T}_h$, we will drop the subscript $h$ whenever the function $\varphi_h$ is evaluated on a triangle $K\in \mathcal{T}_h$, that is, $\varphi_K \coloneqq \varphi_h|_K$.
To address the numerical fluxes at edges of $\mathcal{T}_h$, we define the average $\avg{\cdot}$ and jump $\jump{\cdot}$ of a scalar function $\varphi_h\in \mathcal{P}_{\ell}(\Omega)$ on $e\in \mathcal{E}_h^{\rm int}$ respectively as follows:
\begin{align*}
	\avg{\varphi_h}&\coloneqq
	\begin{cases}
		\dfrac{\varphi_K+\varphi_L}{2}&\text{if } e=\partial K\cap\partial L\in\mathcal{E}_h^{\rm int},\\
		\varphi_K&\text{if }e=e\cap\partial K\in\mathcal{E}_h^{\partial},
	\end{cases}\qquad
	\jump{\varphi_h}\coloneqq
	\begin{cases}
		\varphi_K-\varphi_L&\text{if } e=\partial K\cap\partial L\in\mathcal{E}_h^{\rm int},\\
		\varphi_K&\text{if }e=e\cap\partial K\in\mathcal{E}_h^{\partial}.
	\end{cases}
\end{align*}
For a scalar function $\varphi$, we employ the identity $\varphi = \varphi_{\oplus}-\varphi_{\ominus}$, where $\varphi_{\oplus}$ and $\varphi_{\ominus}$ are the positive and negative parts of $\varphi$, defined as
\begin{align} \label{eq:positive-negative:parts}
	\varphi_\oplus\coloneqq\frac{\vert \varphi\vert +\varphi}{2}=\max\{\varphi,0\},
	\qquad
	\varphi_{\ominus}\coloneqq\frac{\vert \varphi\vert -\varphi}{2}=-\min\{\varphi,0\}.
\end{align}
Note that both functions $\varphi_{\oplus}$ and $\varphi_{\ominus}$ are non-negative.

The unimodal mobility function $M(u)$ defined by~\eqref{eq:M} has its maximizer at $u_{\rm mid} \coloneqq {\rho_{\rm b}}/(2 + \gamma)$.
It is increasing for $u<u_{\rm mid}$ and decreasing otherwise.
We define the following functions:
\begin{subequations}
	\begin{alignat*}{2}
		M^\uparrow(u)&\coloneqq
		\begin{cases}
			M(u) &\text{ if }u\le u_{\rm mid},\\
			M(u_{\rm mid}) &\text{ if }u>u_{\rm mid},
		\end{cases}\\
		M^\downarrow(u)&\coloneqq
		\begin{cases}
			0 &\text{ if }u\le u_{\rm mid},\\
			M(u) - M(u_{\rm mid}) &\text{ if }u >u_{\rm mid}.
		\end{cases}
	\end{alignat*}
\end{subequations}
Observe that $ M^\uparrow(u) + M^\downarrow(u) = M(u)$, and $M^\uparrow(u)\geq 0$ is an increasing function, while $M^\downarrow(u)\leq 0$ is decreasing. 
Both $M^\uparrow$ and $M^\downarrow$ are used to approximate the numerical fluxes of \eqref{eq:total} at the edges of each element $K\in \mathcal{T}_h$ in the upcoming section. 
In addition, to address the numerical flux of \eqref{eq:solids}, we introduce the functions:
\begin{align*}
	M_{\bB{c}}^{\uparrow}(u) & = \begin{cases}
		M_{\bB{c}}(u) &\text{ if }u\le u_{\rm mid},\\
		M(u_{\rm mid})/u &\text{ if }u> u_{\rm mid},
	\end{cases}\\
	M_{\bB{c}}^{\downarrow}(u) & = \begin{cases}
		0 &\text{ if }u\le u_{\rm mid},\\
		M_{\bB{c}}(u)-M(u_{\rm mid})/u &\text{ if }u> u_{\rm mid},
	\end{cases}
\end{align*}
which satisfy $M_{\bB{c}}^{\uparrow}(u)\geq 0$ and $M_{\bB{c}}^{\downarrow}(u)\leq 0$, the sum is such that $M_{\bB{c}}^{\uparrow}(u) + M_{\bB{c}}^{\downarrow}(u)= M_{\bB{c}}(u)$ for all $u\in \R$ and the following identities hold
\begin{align*}
	uM_{\bB{c}}^{\uparrow}(u) = M^{\uparrow}(u),\qquad uM_{\bB{c}}^{\downarrow}(u) = M^{\downarrow}(u) \quad \text{for}\quad u\in[0,\rho_{\rm b}].
\end{align*}
Observe that $M_{\bB{c}}^{\uparrow}$ does not correspond to the increasing part of $M_{\bB{c}}$ and the notation with the arrows in this case is mainly because of the identities given above.

During the rest of the paper and without loss of generality, whenever we refer to a general edge $e\in \mathcal{E}_h^{\rm int}$, we will assume that this edge is shared by generic triangles $K\in \mathcal{T}_h$ and $L\in\mathcal{T}_h$, unless explicitly specified the triangles. Finally, we will make use of the standard notation for the space of square-integrable measurable scalar functions $\mathrm{L}^2(\Omega)$, and Sobolev spaces $\H^k(\Omega)$, with $k$ integer, and their respective norms $\|\cdot\|_{0,\Omega}$ and $\|\cdot\|_{k,\Omega}$.
In turn, we will use bold font for the case of vector spaces. Furthermore, we let $(\cdot,\cdot)_{\Ltwo}$ be the standard inner product in $\mathrm{L}^2(\Omega)$, and introduce the following spaces
\begin{align*}
	\mathbf{H}(\div;\Omega)  &: = \bigg\{\bB{w}\in  \mathbf{L}^2(\Omega):\quad\div(\bB{w})\in \mathrm{L}^2(\Omega) \bigg\},\\
	\H_0^1(\Omega) &:= \Big\{\varphi \in \H^1(\Omega):\quad \varphi = 0\quad\text{on }\Gamma\Big\},
\end{align*}
where the norm of $\mathbf{H}(\div;\Omega)$ is defined as $\|\cdot\|_{\div,\Omega}:=\|\cdot\|_{0,\Omega} + \|\div(\cdot)\|_{0,\Omega}$.

\subsection{Discontinuous Galerkin formulation}\label{sec:scheme:discontinuous}
To approximate Equations~\eqref{syst:main}, we first test each of the four Equations~\eqref{eq:solids}--\eqref{eq:total} against a test function $\varphi_h\in \mathcal{P}_{\ell}(\Omega)$, and test Equation~\eqref{eq:potencial} against $\vartheta_h\in \mathcal{P}_{\ell}^{\rm cont}(\Omega)$, and integrate over $\Omega$. 
Then, applying integration by parts on the divergence operators of each equation in \eqref{syst:main} for all $K\in\mathcal{T}_h$, we obtain the weak formulation:
\begin{subequations} \label{syst:semid:space}
	\begin{alignat}{2}
		\big(\partial_t c^{(i)},\varphi_h\big)_{\Ltwo} + \aop\big(c^{(i)}\big(\bB{q}-M_{\bB{c}}(u)\nabla\mu\big),\varphi_h\big) &= \big( \mathcal{R}_{\bB{c}}^{(i)}(\bB{c},\bB{s}), \varphi_h\big)_{\Ltwo}, \qquad&& i=1,\dots, k_{\rm b},\label{eq:semid:c}\\[1ex]
		\big(\partial_t s^{(j)},\varphi_h\big)_{\Ltwo} + \aop\big(s^{(j)}\big(\bB{q} + M_{\bB{s}}(u)\nabla\mu\big),\varphi_h\big) &= \big( \mathcal{R}_{\bB{s}}^{(j)}(\bB{c},\bB{s}), \varphi_h\big)_{\Ltwo}, \qquad&& j=1,\dots, k_{\rm f}\label{eq:semid:s},\\[1ex]
		\big(\partial_t u,\varphi_h\big)_{\Ltwo} + \aop\big(u \bB{q} - M(u)\nabla\mu,\varphi_h\big) &= \big(\mathcal{\widetilde{R}}(\bB{c},\bB{s}), \varphi_h\big)_{\Ltwo},\label{eq:semid:u}\\[1ex]
		\dfrac{\kappa}{\rho_{\rm b}} \big( \nabla u, \nabla \vartheta_h \big)_{\Ltwo} + \big(\Psi'(u),\vartheta_h\big)_{\Ltwo}  & =  \big(\mu,\vartheta_h)_{\Ltwo},\label{eq:semid:mu}
	\end{alignat}
\end{subequations}
where the total flux $\aop$, which results from the integration by parts on \eqref{eq:semid:c}, \eqref{eq:semid:s} and \eqref{eq:semid:u}, is defined by
\begin{align*}
	\aop(\bB{w},v) & := -\sum_{K\in \mathcal{T}_h} \int_K \bB{w} \cdot\nabla v
	+ \sum_{e\in \mathcal{E}_h^{\rm int}} \int_{e} \big(\bB{w}\cdot \nn_e\big) \, \jump v,
\end{align*}
for all $\bB{w}\in \mathbf{L}^2(\Omega)$ and $v\in \mathrm{L}^2(\Omega)$ such that $v|_K \in \H^1(K)$ for all $K\in \mathcal{T}_h$.
To address the time discretization of \eqref{syst:semid:space}, we consider the time points $t^n = n\Delta t$ with $n\in \mathbb{N}$ and $\Delta t>0$ the time step, and employ the backward Euler method for the approximation of the time derivatives in \eqref{eq:semid:c}--\eqref{eq:semid:u}. 
Furthermore, to handle the reaction terms of \eqref{eq:semid:c}--\eqref{eq:semid:u}, we make use of a splitting procedure.
We define the auxiliary unknowns $\widehat{u}:=\widehat{u}(\bB{x},t)\in \R$, $\widehat{\bB{c}}:=\widehat{\bB{c}}(\bB{x},t)\in \R^{\nsp}$, and $\widehat{\bB{s}}:=\widehat{\bB{s}}(\bB{x},t)\in \R^{\ssp}$, and introduce a single-step approximation of an ODE whose right-hand side is given by the respective reaction terms of \eqref{eq:semid:c}, \eqref{eq:semid:s} and \eqref{eq:semid:u}.

The fully discrete discontinuous Galerkin scheme for \eqref{eq:semid:c}--\eqref{eq:semid:mu} is written as follows:
Given $n\geq 0$ and $(u_{h}^{n},\bB{s}_{h}^{n},\bB{c}_{h}^{n})\in \mathcal{P}_{\ell}(\Omega)\times [\mathcal{P}_{\ell}(\Omega)]^{\ssp}\times [\mathcal{P}_{\ell}(\Omega)]^{\nsp}$,
find $u_{h}^{n+1}, \widehat{u}_h^{\, n+1}\in \mathcal{P}_{\ell}(\Omega)$, $\bB{c}_{h}^{n+1}, \widehat{\bB{c}}_{h}^{n+1}\in [\mathcal{P}_{\ell}(\Omega)]^{\nsp}$, $\bB{s}_{h}^{n+1}, \widehat{\bB{s}}_{h}^{n+1}\in [\mathcal{P}_{\ell}(\Omega)]^{\ssp}$, and $\widetilde{u}^{n+1}_h, \mu_{h}^{n+1}\in \mathcal{P}^{\rm cont}_{\ell+1}(\Omega)$ such that
\begin{subequations} \label{syst:scheme:fulld}
	\begin{align}
		\dfrac{1}{\Delta t}\big(\widehat{u}_h^{\, n+1}- u_h^{n},\varphi_h\big)_{\Ltwo} + \aconv\big(\widehat{u}_h^{\,n+1},\varphi_h; \bB{q}_h^{n}\big) + \adiff\big(\widehat{u}_h^{\, n+1},\varphi_h; -\nabla\mu_h^{n+1}\big) & = 0, \label{eq:fulld:u:hat} \\[1.5ex]
		\dfrac{1}{\Delta t} \big(\widehat{c}^{\, (i),n+1}_h- c^{(i),n}_h,\varphi_h\big)_{\Ltwo} + \aconv\big(\widehat{c}^{\, (i),n+1}_h, \varphi_h; \bB{q}_h^{n}\big)+ \adiffc\big(\widehat{c}^{\, (i),n+1}_h, \varphi_h;  - \nabla\mu_h^{n+1}, \widehat{u}_h^{\, n+1}\big) & = 0, \label{eq:fulld:c:hat}
		\\[1.5ex]
		\dfrac{1}{\Delta t}\big(\widehat{s}^{(j),n+1}_h-s^{(j),n}_h,\varphi_h\big)_{\Ltwo} + \aconv\big(\widehat{s}^{(j),n+1}_h,\varphi_h; \bB{q}_h^{n}\big) + \adiffs\big(\widehat{s}^{(j),n+1}_h,\varphi_h; \nabla\mu_h^{n+1}, \widehat{u}_h^{\, n+1}\big) & = 0, \label{eq:fulld:s:hat}
	\end{align}
	coupled to
	\begin{align}
		\big(c^{(i),n+1}_h,\varphi_h\big)_{\Ltwo} & = \big(\widehat{c}^{\, (i),n+1}_h,\varphi_h\big)_{\Ltwo} +  \Delta t\big(\mathcal{R}_{\bB{c}}^{(i)}(\widehat{\bB{c}}_h^{n+1},\widehat{\bB{s}}_h^{n+1}),\varphi_h\big)_{\Ltwo}, \label{eq:fulld:c}\\[1ex]
		\big(s^{(j),n+1}_h,\varphi_h\big)_{\Ltwo} & = \big(\widehat{s}^{(j),n+1}_h,\varphi_h\big)_{\Ltwo} +  \Delta t\big(\mathcal{R}_{\bB{s}}^{(j)}(\widehat{\bB{c}}_h^{n+1},\widehat{\bB{s}}_h^{n+1}),\varphi_h\big)_{\Ltwo}, \label{eq:fulld:s}\\[1ex]
		\big(u^{n+1}_h,\varphi_h\big)_{\Ltwo} & = \big(\widehat{u}^{\, n+1}_h,\varphi_h\big)_{\Ltwo} +  \Delta t\big(\mathcal{\widetilde{R}}(\widehat{\bB{c}}_h^{n+1},\widehat{\bB{s}}_h^{n+1}),\varphi_h\big)_{\Ltwo}, \label{eq:fulld:u}\\[1ex]
		(\widetilde{u}_h^{n+1},\vartheta_h)_{\Ltwo} & = (u_h^{n+1},\vartheta_h)_{\Ltwo},\label{eq:other:u}\\[1ex]
		\big(\mu_h^{n+1},\vartheta_h\big)_{\Ltwo} & = \dfrac{\kappa}{\rho_{\rm b}}\big(\nabla \widetilde{u}_h^{n+1}, \nabla \vartheta_h\big)_{\Ltwo} + \big(\Psi'(u_h^{n+1}),\vartheta_h\big)_{\Ltwo} , \label{eq:fulld:mu}
	\end{align}
\end{subequations}
for $i=1,\dots,\nsp$ and $j=1,\dots,\ssp$, and for all test functions $\varphi_h\in\mathcal{P}_{\ell}(\Omega)$ and $\vartheta_h\in\mathcal{P}_{\ell+1}^{\rm cont}(\Omega)$, where the discrete bilinear forms are described below. 
Given $\bB{\beta}\in \mathbf{H}^1(\Omega)$ or $\bB{\beta}\in \mathbf{H}({\rm div};\Omega)$, the upwind operator $\aconv$ approximating the convective terms at \eqref{syst:scheme:fulld} is defined by
\begin{align*}
	\aconv\big(\varphi_h,w_h; \bbeta\big)&:= - \big( \varphi_h\bbeta,\nabla w_h\big)_{\Ltwo}
	+ \sum_{e\in\mathcal{E}_h^{\rm int}}\int_e\upwF(\varphi_K, \varphi_L; \bB{\beta}) \jump{w_h}, \\
	\upwF(\varphi_K, \varphi_L; \bB{\beta})  &:=  (\bB{\beta}\cdot\nn_e)_{\oplus}\varphi_K - (\bB{\beta}\cdot\nn_e)_{\ominus}\varphi_L,
\end{align*}
for $\varphi_h,w_h\in \mathcal{P}_{\ell}(\Omega)$.
Next, to define the operators $\adiff$, $\adiffc$ and $\adiffs$, we consider a vector function $\bB{\beta}\in \mathbf{L}^2(\Omega)$, possibly discontinuous across edges in $\mathcal{E}_h$, and introduce the following numerical fluxes for each $e\in \mathcal{E}^{\rm int}_h$
\begin{equation*}
	\uF(\varphi_K, \varphi_L; \bB{\beta}) :=  \bar{\beta}^e_{\oplus}\Big\{\Mup(\varphi_K) + \Mdn(\varphi_L) \Big\}
	- \bar{\beta}^e_{\ominus}\Big\{\Mup(\varphi_L) + \Mdn(\varphi_K)\Big\},
\end{equation*}
and
\begin{align*}
	\cF(\varphi_K, \varphi_L; \bB{\beta},\widetilde{\varphi}_h)  &:=  \bar{\beta}^e_{\oplus}\Big\{\MupC(\widetilde{\varphi}_K)(\varphi_K)_{\oplus} + \MdnC(\widetilde{\varphi}_L)(\varphi_L)_{\oplus} \Big\}\\
	&\qquad - \bar{\beta}^e_{\ominus}\Big\{\MupC(\widetilde{\varphi}_L)(\varphi_L)_{\oplus} + \MdnC(\widetilde{\varphi}_K)(\varphi_K)_{\oplus}\Big\},\\[1ex]
	\sF(\varphi_K, \varphi_L; \bB{\beta},\widetilde{\varphi}_h)  &:=  \bar{\beta}^e_{\oplus}M_{\bB{s}}(\widetilde{\varphi}_K)(\varphi_K)_{\oplus}  - \bar{\beta}^e_{\ominus}M_{\bB{s}}(\widetilde{\varphi}_L)(\varphi_L)_{\oplus},
\end{align*}
for all $\varphi_h, \widetilde{\varphi}_h \in \mathcal{P}_{\ell}(\Omega)$, with the short-hand notation $\bar{\beta}^e\coloneqq \avg{\bbeta}\cdot\nn_e$. Note that the three fluxes $\upwF$, $\uF$, $\cF$ and $\sF$ are upwind numerical fluxes in $\mathcal{P}_\ell(\Omega)$, where $\upwF$ is defined to approximate the convective parts in \eqref{eq:fulld:c:hat}, \eqref{eq:fulld:s:hat} and \eqref{eq:fulld:u:hat}, and $\uF$, $\cF$ and $\sF$ are related to diffusive terms.
Then, the respective upwind operators are defined as
\begin{align}
	\adiff\big(\varphi_h,w_h; \bB{\beta}\big) :=& \, -\big(M(\varphi_h) \bbeta,\nabla w_h \big)_{\Ltwo}
	+\sum_{e\in \mathcal{E}_h^{\rm int}} \int_{e} \uF(\varphi_K, \varphi_L; \bB{\beta})\jump{w_h},notag\\
	b_h^{\bB{\xi}}\big(\varphi_h,w_h; \bbeta, \widetilde{\varphi}_h\big) := &\, -\big(\varphi_h M_{\bB{\xi}}(\widetilde{\varphi}_h)\bbeta,\nabla w_h \big)_{\Ltwo}
	+ \sum_{e\in \mathcal{E}_h^{\rm int}} \int_e\widetilde{\Phi}_e^{\bB{\xi}}\big(\varphi_K,\varphi_L;\bB{\beta},\widetilde{\varphi}_h)\jump{w_h}, \label{def:op:c:s}
\end{align}
for all $\varphi_h, \widetilde{\varphi}_h, w_h \in \mathcal{P}_{\ell}(\Omega)$ and $\bB{\xi}\in\{\bB{c},\bB{s}\}$.
Observe that instead of strongly imposing that the sum of the components of vector $\bB{c}$ should be equal to the total concentration $u$ at the time step $t^{n+1}$, we built  the numerical flux $\adiffc$ in such a way that this property is in fact preserved. 
The latter is shown in Theorem~\ref{thm:maxprinciple:ODEs}.

For the Stokes system \eqref{eq:stokesmom}--\eqref{eq:stokesdiv}, a suitable approximation can be made of a stable pair of polynomial subspaces for the velocity $\bB{q}$ and pressure $p$, such that the normal component of the discrete velocity is continuous across edges in $\mathcal{E}_h$, and Equation \eqref{eq:stokesdiv} is satisfied on each $K\in \mathcal{T}_h$.
The continuity constraint can be achieved seeking for approximations of $\bB{q}$ in $\mathbf{H}(\div; \Omega)$ or $\mathbf{H}^1(\Omega)$. 
For this reason, we propose a discrete formulation of \eqref{eq:stokesmom}--\eqref{eq:stokesdiv} based on the primal formulation given in \cite[Section 8.4.3]{Boffi2013}, that is, an $H^1$-conforming approximation for $\bB{q}$ in the subspace $\spaceq \coloneqq \mathbcal{P}_{2}(\Omega)\cap \mathbf{H}_0^1(\Omega)$, and a piecewise constant approximation for $p$ in $\spacep\coloneqq\mathcal{P}_0(\Omega)\cap \L_0^2(\Omega)$. 
Then, the discrete formulation of the Stokes system is given as follows: Given $n\geq 0$, find $(\bB{q}_h^{n},p_h^{n})\in \spaceq\times \spacep$ such that
\begin{subequations} \label{syst:fulld:stokes}
	\begin{alignat}{2}
		-\big( \nu(u_h^{n})\bB{e}(\bB{q}_h^{n}), \bB{e}(\bB{w}_h) \big)_{\Ltwo} + \big(p_h^{n},\div(\bB{w}_h)\big)_{\Ltwo} & = \big(\mathbf{f}_h^{n}, \bB{w}_h\big)_{\Ltwo},&& \qquad \forall \bB{w}_h\in \spaceq, \label{eq:fulld:q}\\
		\big( \varphi_h,\div(\bB{q}_h^{n})\big)_{\Ltwo} &  = 0, && \qquad \forall \varphi_h\in \spacep, \label{eq:fulld:divq}
	\end{alignat}
\end{subequations}
where the discretized source term \eqref{eq:source:stokes} is given by
\begin{align}\label{eq:discrete:source:f}
	\mathbf{f}_h^{n} = -\dfrac{g\Delta\rho}{\rhoC}u_h^{n}\bB{k}+ \eta\Big(-\dfrac{\kappa}{\rhoC}\nabla^2 \widetilde{u}_h^{n} +  \Psi'(u_h^{n})\Big)\nabla \widetilde{u}_h^{n}.
\end{align}
Observe that a pseudo-stress formulation with $\mathbb{S}=\mathbb{S}(u,\bm{v},p)$ as an additional unknown (see, e.g., \cite{Gatica2010,Oyarzua2023}) can also be proposed for \eqref{eq:stokesmom}--\eqref{eq:stokesdiv}. However, to ensure Equation~\eqref{eq:stokesdiv} on each element $K\in\mathcal{T}_h$ with $\bB{q}_h^{n}\in \mathbf{H}(\div;\Omega)$, an extended system is required, see for instance \cite{Oyarzua2023}.
In practice, the incompressibility condition~\eqref{eq:stokesdiv}, which is present in the definition of $\L_0^2(\Omega)$, is implemented via Lagrange multipliers.

\subsection{Properties of the scheme}\label{sec:scheme:properties}
In this section, we provide an invariant-region property of the concentrations in the scheme \eqref{syst:scheme:fulld} and the well posedness for the discretized Stokes system~\eqref{syst:fulld:stokes}.
The positivity-preserving results are primarily based on the work by Acosta-Soba et al.~\cite{Acosta2023b}, wherefore we will only analyse the case $\ell=0$.
We first collect some useful relations.
Given $w,\varsigma,\alpha\in \R$ and referring to the positive and negative parts of a scalar function defined by~\eqref{eq:positive-negative:parts}, we have
\begin{align}
	w_{\oplus}w_{\ominus} & = 0,\qquad
	w (w)_\ominus  = -(w_\ominus)^2, \qquad
	w (w)_\oplus  = (w_\oplus)^2, \label{eq:prop:o}
\end{align}
and we can straightforwardly show that
\begin{align}
	\big(w - \varsigma \big)w_\ominus &\le 0\quad\text{for}\quad \varsigma \geq 0. \label{eq:ineq:o:1}
\end{align}
In addition, given $\varphi_h\in \mathcal{P}_0(\Omega)$, we let $K_{\circ}, K^{\diamond} \in \mathcal{T}_h$ be two elements such that
\begin{align*}
	\varphi_{K_{\circ}} = \min_{L\in\mathcal{T}_h}\{\varphi_L\},\qquad \varphi_{K^{\diamond}} = \max_{L\in\mathcal{T}_h}\{\varphi_L\}.
\end{align*}
Note that $K_{\circ}$ and $K^{\diamond}$ may not be uniquely defined, in which case we will consider those with the highest index in an ordered mesh $\mathcal{T}_h$.
Then, given $\alpha\in \R$, we define the operators $\Pi_{\ominus}^{\alpha}:\mathcal{P}_0(\Omega)\to \mathcal{P}_0(\Omega)$ and $\Pi_{\oplus}^{\alpha}:\mathcal{P}_0(\Omega)\to \mathcal{P}_0(\Omega)$ such that for all $\varphi_h\in \mathcal{P}_0(\Omega)$
\begin{align*}
	\Pi_{\ominus}^\alpha(\varphi_h) =
	\begin{cases}
		(\varphi_{K_{\circ}} - \alpha)_{\ominus}&\text{in }K_{\circ},\\
		0&\text{otherwise},
	\end{cases}\qquad
	\Pi_{\oplus}^{\alpha}(\varphi_h) =
	\begin{cases}
		(\varphi_{K^{\diamond}} - \alpha)_{\oplus}&\text{in }K^{\diamond},\\
		0&\text{otherwise}.
	\end{cases}
\end{align*}
In the following lemma, we introduce some useful properties of $\aconv$ and $\adiff$, which will be employed when showing the maximum principle for the discrete concentrations.

\begin{lemma}\label{lem:auxiliar}
	Given $\bB{\beta}\in \mathbf{H}(\div;\Omega)$ or $\bB{\beta}\in \mathbf{H}^1(\Omega)$, such that $\div(\bB{\beta}) = 0$ in $\Omega$, and given $\varphi_h\in \mathcal{P}_0(\Omega)$, and $\alpha \in \mathbb{R}$ constant, then
	\begin{align} \label{ineq:aux:aconv}
		\aconv\big(\varphi_h, \Pi_{\ominus}^{\alpha}(\varphi_h); \bB{\beta}\big) \le 0\qquad\text{and}\qquad
		\aconv\big(\varphi_h, \Pi_{\oplus }^{\alpha}(\varphi_h); \bB{\beta}\big) \ge 0.
	\end{align}
\end{lemma}
\begin{proof}
	See Appendix \ref{appendix:A}.
\end{proof}

Next, we show the existence of solutions to the discretized extended Cahn--Hilliard system \eqref{syst:scheme:fulld}. 
This result is, in fact, an extension to \cite[Theorem 3.13]{Acosta2023a}, and therefore following the reasoning therein exposed we can show the following theorem.

\begin{theorem} \label{thm:existence:u}
	Given the solution $\bB{q}^n_h\in \spaceq$ to~\eqref{syst:fulld:stokes}, and $\bB{c}^{n}_h\in [\mathcal{P}_0(\Omega)]^{\nsp}$, $\bB{s}_h^n\in [\mathcal{P}_0(\Omega)]^{\ssp}$ and $u_h^{\, n}\in \mathcal{P}_0(\Omega)$ such that $\bB{0}\le \bB{c}^n_h\le \rho_{\rm b}\bB{1}$, $\bB{s}_h^n\geq\bB{0}$ and $0\leq u^{\, n}_h\leq \rho_{\rm b}$ in $\overline\Omega$, then there exists a solution to the scheme \eqref{syst:scheme:fulld}.
\end{theorem}
\begin{proof}
The proof is based on that of \cite[Theorem 3.13]{Acosta2023b}. For the ease of reading, we drop the superscript $n+1$ of the unknowns, and introduce the following product spaces
\begin{align*}
	\mathbf{Q}_h^{0} := \mathcal{P}_0(\Omega)\times [\mathcal{P}_0(\Omega)]^{\nsp}\times[\mathcal{P}_0(\Omega)]^{\ssp}\quad\text{and}\quad \mathbf{Q}_h := \mathbf{Q}_h^{0}\times \mathcal{P}_1^{\rm cont}(\Omega)\times \mathcal{P}_1^{\rm cont}(\Omega).
\end{align*}
Then, we define the discrete operator $\mathbf{T}_h:\mathbf{Q}_h^{0}\times\mathbf{Q}_h\mapsto \mathbf{Q}_h^{0}\times\mathbf{Q}_h
$ such that for every $\widehat{\bB{u}}_h^* = (\widehat{u}_h^*,\widehat{\bB{c}}_h^*,\widehat{\bB{s}}_h^*)\in \mathbf{Q}_h^0$ and $\bB{u}_h^* = \big(u_h^*,\bB{c}_h^*,\bB{s}_h^*,\widetilde{u}_h^*,\mu_h^*\big)\in \mathbf{Q}_h$, there holds
\begin{align*}
	\mathbf{T}_h(\widehat{\bB{u}}_h^*,\bB{u}_h^*) = (\widehat{\bB{u}}_h,\bB{u}_h)\in \mathbf{Q}_h^{0}\times\mathbf{Q}_h,
\end{align*}
where $\widehat{\bB{u}}_h = (\widehat{u}_h,\widehat{\bB{c}}_h,\widehat{\bB{s}}_h)\in \mathbf{Q}_h^0$ and
$\bB{u}_h = \big(u_h,\bB{c}_h,\bB{s}_h,\widetilde{u}_h,\mu_h\big)\in \mathbf{Q}_h$ are the components of the solution to the following partially linearized system:
\begin{subequations}\label{syst:FP}
	\begin{align}
		\dfrac{1}{\Delta t}\big(\widehat{u}_h - u_h^{n},\varphi_h\big)_{\Ltwo} + \aconv\big(\widehat{u}_h,\varphi_h; \bB{q}_h^{n}\big) & =  - \adiff\big(\widehat{u}_h^*,\varphi_h; -\nabla\mu_h^*\big), \label{eq:FP:u:hat}
		\\
		\dfrac{1}{\Delta t}\big(\widehat{c}^{\, (i)}_h - c^{(i),n}_h,\varphi_h\big)_{\Ltwo} + \aconv\big(\widehat{c}^{\, (i)}_h, \varphi_h; \bB{q}_h^{n}\big) & = - \adiffc\big(\widehat{c}^{\,(i),*}_h, \varphi_h;  - \nabla\mu_h^*, \widehat{u}_h^*\big), \label{eq:FP:c:hat}
		\\
		\dfrac{1}{\Delta t} \big(\widehat{s}^{(j)}_h - s^{(j),n}_h,\varphi_h\big)_{\Ltwo} + \aconv\big(\widehat{s}^{(j)}_h,\varphi_h; \bB{q}_h^{n}\big) & =  - \adiffs\big(\widehat{s}^{(j),*}_h,\varphi_h; \nabla\mu_h^*, \widehat{u}_h^*\big)
		, \label{eq:FP:s:hat}
	\end{align}
	with $i = 1,\dots,\nsp$ and $j=1,\dots,\ssp$, coupled to
	\begin{align}
		\bB{c}_K & = \widehat{\bB{c}}_K +  \Delta t\, \mathbcal{R}_{\bB{c}}\big(\widehat{\bB{c}}_K,\widehat{\bB{s}}_K\big), \label{eq:FP:c}\\[1ex]
		\bB{s}_K & = \widehat{\bB{s}}_K +  \Delta t\, \mathbcal{R}_{\bB{s}}\big(\widehat{\bB{c}}_K,\widehat{\bB{s}}_K\big), \label{eq:FP:s}\\[1ex]
		u_K & = \widehat{u}_K +  \Delta t\, \mathcal{\widetilde{R}}(\widehat{\bB{c}}_K,\widehat{\bB{s}}_K), \label{eq:FP:u}\\
		(\widetilde{u}_h,\vartheta_h)_{\Ltwo} &  = (u_h,\vartheta_h)_{\Ltwo},\label{eq:FP:other:u}\\
		\big(\mu_h,\vartheta_h\big)_{\Ltwo}  & =   \dfrac{\kappa}{\rho_{\rm b}}\big(\nabla \widetilde{u}_h, \nabla \vartheta_h\big)_{\Ltwo} + \big(\Psi'(u_h),\vartheta_h\big)_{\Ltwo}, \label{eq:FP:mu}
	\end{align}
\end{subequations}
for all $K\in \mathcal{T}_h$, $\varphi_h\in \mathcal{P}_0(\Omega)$ and $\vartheta_h\in \mathcal{P}_1^{\rm cont}(\Omega)$. 
As mentioned in \cite{Acosta2023b}, it is easy to show that for $\bB{q}_h^n\in \mathcal{P}_{2,0}^{\rm cont}$  such that $\div(\bB{q}_h^n)=0$ in $\Omega$, there exist unique $\widehat{u}_h$, $\widehat{\bB{c}}_h$ and $\widehat{\bB{s}}_h$ to \eqref{eq:FP:u:hat}, \eqref{eq:FP:c:hat} and \eqref{eq:FP:s:hat}, respectively. 
It is straightforward to observe that there exist unique solutions $\bB{c}_h$, $\bB{s}_h$ and $u_h$ solutions to~\eqref{eq:FP:c}, \eqref{eq:FP:s} and \eqref{eq:FP:u}, respectively, and subsequently there exists a unique solution $\widetilde{u}_h\in \mathcal{P}_1^{\rm cont}(\Omega)$ to~\eqref{eq:FP:other:u}.
Then, there exists a unique solution $\mu_h\in \mathcal{P}_1^{\rm cont}(\Omega)$ to~\eqref{eq:FP:mu} and the operator $\mathbf{T}_h$ is well defined. 
The continuity of $\mathbf{T}_h$ can be proved with the same arguments used in the proof of \cite[Theorem 3.13]{Acosta2023b}, and since $\mathbf{Q}_h^0$ and $\mathbf{Q}_h$ are finite-dimensional spaces, $\mathbf{T}_h$ is compact. 

Next, we define the following set
\begin{align*}
	\mathfrak{B}\coloneqq\Big\{(\widehat{\bB{u}}_h,\bB{u}_h)\in \mathbf{Q}_h^0\times \mathbf{Q}_h: \quad (\widehat{\bB{u}}_h,\bB{u}_h)=\alpha \mathbf{T}_h(\widehat{\bB{u}}_h,\bB{u}_h)\quad \text{ for some } 0\le\alpha\le1\Big\},
\end{align*}
which we will prove is bounded independently of $\alpha$. 
Assuming $\alpha\in(0,1]$, omitting the trivial case $\alpha=0$, we observe that if $(\widehat{\bB{u}}_h,\bB{u}_h)\in \mathfrak{B}$, then $\widehat{u}_h\in\mathcal{P}_h(\Omega)$ is the solution to
\begin{align} \label{eq:aux:FP:uhat}
	\dfrac{1}{\Delta t}\big(\widehat{u}_h - \alpha u_h^{n},\varphi_h\big)_{\Ltwo} + \aconv\big(\widehat{u}_h,\varphi_h; \bB{q}_h^{n}\big) & =  - \adiff\big(\widehat{u}_h,\varphi_h; -\nabla\mu_h\big),\qquad \forall \varphi_h\in \mathcal{P}_h(\Omega).
\end{align}
Now, in the same way as in Lemma~\ref{lem:maxprinciple:u}, we can conclude that the solution to \eqref{eq:aux:FP:uhat} satisfies that $\widehat{u}_h\geq 0$. 
Then, replacing $\varphi_h = 1\in\mathcal{P}_0(\Omega)$ into \eqref{eq:aux:FP:uhat}, we have
\begin{align*}
	\|\widehat{u}_h\|_{\mathrm{L}^1(\Omega)} = \int_\Omega \widehat{u}_h = \alpha \int_\Omega u_h^n\,\leq\, \|u_h^n\|_{\mathrm{L}^1(\Omega)}.
\end{align*}

Analogously, following the proofs of Lemmas~\ref{lem:maxprinciple:c} and \ref{lem:maxprinciple:s}, we can show that the solutions to \eqref{eq:FP:c:hat} and \eqref{eq:FP:s:hat} fulfil $\widehat{c}_h^{(i)}\geq 0$ and $\widehat{s}_h^{(j)}\geq 0$, respectively, for all $i=1,\dots,\nsp$ and $j = 1,\dots,\ssp$. 
Therefore, from Equations~\eqref{eq:FP:c:hat} and \eqref{eq:FP:s:hat}, we get the respective bounds:
\begin{align*}
	\|\widehat{c}_h^{(i)}\|_{\mathrm{L}^1(\Omega)} \leq \|c_h^{(i),n}\|_{\mathrm{L}^1(\Omega)}\quad\text{and}\quad
	\|\widehat{s}_h^{(j)}\|_{\mathrm{L}^1(\Omega)} \leq \|s_h^{(j),n}\|_{\mathrm{L}^1(\Omega)}.
\end{align*}
In addition, adding up over all $K\in\mathcal{T}_h$ each Equation~\eqref{eq:FP:c}, \eqref{eq:FP:s} and \eqref{eq:FP:u}, respectively, and using Assumption~\ref{asumption:reactionterms}, we have
\begin{align*}
	\|c_h^{(i)}\|_{\mathrm{L}^1(\Omega)}  &\leq \|c_h^{(i),n}\|_{\mathrm{L}^1(\Omega)} + \Delta t\,|\Omega|\boundr_{\max},\\
	\|s_h^{(j)}\|_{\mathrm{L}^1(\Omega)} &\leq \|s_h^{(j),n}\|_{\mathrm{L}^1(\Omega)} + \Delta t\,|\Omega|\boundr_{\max},\\
	\|u_h\|_{\mathrm{L}^1(\Omega)} &\leq \|u_h^{n}\|_{\mathrm{L}^1(\Omega)} + \Delta t\,|\Omega|\boundr_{\max}.
\end{align*}
Moreover, $\widetilde{u}_h\in \mathcal{P}_1^{\rm cont}(\Omega)$ is the solution to~\eqref{eq:FP:other:u}.
Setting $\vartheta_h = \widetilde{u}_h$ and using that $0\leq u_h\leq \rhoC$ (see for instance Theorem~\ref{thm:maxprinciple:ODEs}), we obtain
\begin{align*}
	\|\widetilde{u}_h\|_{\Ltwo}^2 \,=\, (u_h , \widetilde{u}_h)_{\Ltwo} \le \rhoC\|\widetilde{u}_h\|_{\mathrm{L}^1(\Omega)}
	\le \rhoC\vert \Omega\vert ^{1/2} \|\widetilde{u}_h\|_{\Ltwo},
\end{align*}
which directly implies the bound $\|\widetilde{u}_h\|_{\Ltwo} \le \rhoC\vert \Omega\vert^{1/2}$.
Into \eqref{eq:FP:mu}, we substitute $\vartheta_h = \mu_h$, where $\mu_h\in \mathcal{P}_h^{\rm cont}(\Omega)$ is the solution to~\eqref{eq:FP:mu}, use the Cauchy--Schwarz inequality and the equivalence of norms in finite-dimensional spaces to get
\begin{align*}
	\|\mu_h\|_{\Ltwo}^2 &
	\le \dfrac{\kappa}{\rho_{\rm b}}\|\nabla \widetilde{u}_h\|_{\Ltwo}\|\nabla \mu_h\|_{\Ltwo} + \|\Psi'(u_h)\|_{\Ltwo}\|\mu_h\|_{\Ltwo}\\
	&\le \dfrac{\kappa}{\rho_{\rm b}}\|\widetilde{u}_h\|_{1,\Omega}\|\mu_h\|_{1,\Omega}+\|\Psi'(u_h)\|_{\Ltwo}\|\mu_h\|_{\Ltwo}\\
	&\le \Big(\delta^2 \dfrac{\kappa}{\rho_{\rm b}} \|\widetilde{u}_h\|_{\Ltwo}+\|\Psi'(u_h)\|_{\Ltwo}\Big)\|\mu_h\|_{\Ltwo},
\end{align*}
where the constant $\delta>0$ is related to the equivalence between the norms $\|\cdot\|_{1,\Omega}$ and $\|\cdot\|_{\Ltwo}$ in the finite-dimensional subspace $\mathcal{P}_1^{\rm cont}(\Omega)$. 
Then, since $0\le u_h\le \rhoC$, and using the bound $|\Psi'(u_h)|\leq \Psi_{\max}'$ (cf.~\eqref{eq:bound:dPsi}), we have
\begin{align*}
	\|\mu_h\|_{\Ltwo} \le  \delta^2\dfrac{\kappa}{\rho_{\rm b}}\|\widetilde{u}_h\|_{\Ltwo}
	+|\Omega|^{1/2}\Psi'_{\max}.
\end{align*}
Finally, since in $\mathcal{P}_0(\Omega)$ and $\mathcal{P}_1^{\rm cont}(\Omega)$ all norms are equivalent, we can conclude that $\mathfrak{B}$ is bounded. 
Thus, using the Leray--Schauder fixed point theorem \cite[Theorem 2.6]{Acosta2023b}, we conclude that there exists a solution $(\widehat{\bB{u}}_h, \bB{u})\in \mathbf{Q}_h^0\times \mathbf{Q}_h$ to the scheme \eqref{syst:scheme:fulld}.
\end{proof}

In what follows, we will show the invariant-region-preserving properties for the auxiliary unknowns $\widehat{u}_h^{n+1}$, $\widehat{\bB{c}}_h^{n+1}$ and $\widehat{\bB{s}}_h^{n+1}$, and for the actual concentration unknown solutions $u_h^{n+1}$, $\bB{c}_h^{n+1}$ and $\bB{s}_h^{n+1}$ to~\eqref{syst:scheme:fulld}. 
We start with the auxiliary solution $\widehat{u}_h^{n+1}\in \mathcal{P}_0(\Omega)$ to~\eqref{eq:fulld:u:hat}.
Its boundedness is shown in \cite[Theorem 3.11]{Acosta2023b} by making use of test functions $\varphi_h = \Pi_{\ominus}^{0}(\widehat{u}_h^{n+1})$ and $\varphi_h = \Pi_{\oplus}^{\rhoC}(\widehat{u}_h^{n+1})$ (in our notation) and further bounds on $\aconv$ and $\adiff$. 
Thus, we omit the proof of the next lemma.

\begin{lemma} \label{lem:maxprinciple:u}
	Given the solution $\bB{q}^n_h\in \spaceq$ to~\eqref{syst:fulld:stokes}, $\mu_{h}^{n+1}\in \mathcal{P}^{\rm cont}_{1}(\Omega)$ and $u^n_h\in\mathcal{P}_0(\Omega)$ such that $0\le u^n_h\le \rho_{\rm b}$ in $\overline\Omega$, then $\widehat{u}^{\, n+1}_h$ computed from Equation~\eqref{eq:fulld:u:hat} satisfies $0\le \widehat{u}^{\, n+1}_h\le \rho_{\rm b}$ in $\overline\Omega$.
\end{lemma}
\begin{proof}
	See \cite[Theorem 3.11]{Acosta2023b}.
\end{proof}

For the auxiliary solution $\widehat{\bB{c}}_h^{n+1}\in [\mathcal{P}_0(\Omega)]^{\nsp}$ to~\eqref{eq:fulld:c:hat}, we will show upper and lower bounds of the components of $\widehat{\bB{c}}_h^{n+1}$, and that the sum of the components of $\widehat{\bB{c}}_h^{n+1}$ sum up to $\widehat{u}_h^{n+1}$. 
To prove the summation property, we state the following lemma.

\begin{lemma}\label{lem:aux:existence}
Given the solution $\bB{q}^n_h\in\spaceq$ to~\eqref{syst:fulld:stokes}, $\mu_{h}^{n+1}\in \mathcal{P}^{\rm cont}_{1}(\Omega)$ and $\widehat{u}_h^{\, n+1}\in \mathcal{P}_0(\Omega)$, then the following equation
\begin{align}\label{eq:aux:linear}
	&\dfrac{1}{\Delta t}\big(\vartheta_h,\varphi_h\big)_{\Ltwo} +  \aconv\big(\vartheta_h,\varphi_h; \bB{q}_h^{n}\big) + \adiffc\big(\vartheta_h, \varphi_h; - \nabla\mu_h^{n+1}, \widehat{u}_h^{\, n+1}\big) = 0,\qquad \forall \varphi_h \in \mathcal{P}_0(\Omega),
\end{align}
has the unique solution $\vartheta_h \equiv 0\in \mathcal{P}_0(\Omega)$.
\end{lemma}

\begin{proof}
We first observe that $\vartheta_h\equiv 0 \in \mathcal{P}_0(\Omega)$ trivially fulfils~\eqref{eq:aux:linear}. 
To show uniqueness of solutions, we follow the same idea as in~\cite[Proposition 26.1]{Eymard2000}. Grouping the flux terms in \eqref{eq:aux:linear}, we can rewrite that equation equivalently as
\begin{equation*}
	\sum_{K \in \mathcal{T}_h}
	\abs{K} \vartheta_{K} \varphi_{K}
	+
	\Delta t \sum_{\substack{e \in \mathcal{E}_h^{\rm int}}}
	\mathcal{G}_{e}(\vartheta_{K}, \vartheta_{L})
	\jump{\varphi_h}
	= 0,\qquad \forall \varphi_h\in \mathcal{P}_0(\Omega),
\end{equation*}
where the total flux function $\mathcal{G}_e$, given by the sum of flux terms arising in $\aconv$ and $\adiffc$ at the edge $e\in \mathcal{E}_h^{\rm int}$, is defined by
\begin{align*}
	\mathcal{G}_{e}(\vartheta_{K}, \vartheta_{L}) \,&\coloneqq \, \int_e \upwF\big(\vartheta_{K}, \vartheta_{L}; \bB{q}^n_h\big)  + \abs{e}\cF\big(\vartheta_{K}, \vartheta_{L};-\nabla\mu_h^{n+1}, \widehat{u}^{\, n+1}_h\big)
	\, = \,  \Upsilon^+_e \vartheta_{K} - \Upsilon^-_e\vartheta_{L},
\end{align*}
with side terms given by
\begin{align*}
	\Upsilon^+_e & \coloneqq \int_e (\bB{q}^{n}_h\cdot\nn_e)_{\oplus} \,+\, \abs{e} \Big\{\bar{\beta}^e_{\oplus}\MupC(\widehat{u}^{n+1}_{K}) \,-\, \bar{\beta}^e_{\ominus}\MdnC(\widehat{u}^{n+1}_{K}) \Big\}, \\[1ex]
	\Upsilon^-_e &\coloneqq \int_e (\bB{q}^{n}_h\cdot\nn_e)_{\ominus} \,-\, \abs{e} \Big\{\bar{\beta}^e_{\oplus}\MdnC(\widehat{u}^{n+1}_{L}) \,+\, \bar{\beta}^e_{\ominus}\MupC(\widehat{u}^{n+1}_{L}) \Big\},
\end{align*}
where $ \bar{\beta}^e = \avg{-\nabla\mu_h^{n+1}}\cdot\nn_e$. Given that $\bB{q}_h^{n}$ and $\nabla \mu_h^{n+1}$ are bounded (as they are piecewise polynomial functions on a bounded set), it is easy to observe that $\mathcal{G}_{e}$ is Lipschitz continuous in its first and second arguments.
The uniqueness of solutions follows directly by reproducing the first part of the proof of \cite[Proposition 26.1]{Eymard2000}.
\end{proof}

In the next lemma, we establish the boundedness of $\widehat{\bB{c}}^{n+1}$, and the property that the sum of its components is equal to $\widehat{u}_h^{n+1}$.

\begin{lemma} \label{lem:maxprinciple:c}
Given the solution $\bB{q}^n_h\in \spaceq$ to~\eqref{syst:fulld:stokes}, $\mu_{h}^{n+1}\in \mathcal{P}^{\rm cont}_{1}(\Omega)$,  $\bB{c}^{n}_h\in [\mathcal{P}_0(\Omega)]^{\nsp}$ and $\widehat{u}_h^{\, n+1}\in \mathcal{P}_0(\Omega)$ such that $\bB{0}\le \bB{c}^n_h\le \rho_{\rm b}\bB{1}$ and $0\le \widehat{u}^{\, n+1}_h\le \rho_{\rm b}$ in $\overline\Omega$, then $\widehat{\bB{c}}^{n+1}_h \in \mathcal{P}_0(\Omega)$ computed from Equation~\eqref{eq:fulld:c:hat} satisfies
\begin{align*}
	\sum_{i=1}^{\nsp} \widehat{c}^{\, (i),n+1}_h =  \widehat{u}_h^{\, n+1}\quad\text{and}\quad \bB{0}\leq \widehat{\bB{c}}^{n+1}_h\leq\rho_{\rm b}\bB{1}\quad \text{in}\quad \overline\Omega,
\end{align*}
where $\bB{1}\in\mathbb{R}^{\kC}$ is the vector of ones.
\end{lemma}

\begin{proof}
The positivity of the vector $\widehat{\bB{c}}_h^{n+1}$ is shown following similar arguments as those used to show Lemma~\ref{lem:maxprinciple:u} in \cite{Acosta2023b}.
For the ease of reading, we define $\psi^{\uparrow}_h \coloneqq M_{\bB{c}}^{\uparrow}(\widehat{u}_h^{\, n+1})$ and $\psi^{\downarrow}_h \coloneqq M_{\bB{c}}^{\downarrow}(\widehat{u}_h^{\, n+1})$, for which we deliberately omitted the superscript $n+1$. 
We observe that both $\psi^{\uparrow}_h$ and $\psi^{\downarrow}_h$ belong to $\mathcal{P}_0(\Omega)$ and satisfy
\begin{align*}
	\psi_h:=\psi_h^{\uparrow} + \psi_h^{\downarrow} = M_{\bB{c}}(\widehat{u}_h^{\, n+1}) \in\mathcal{P}_0(\Omega).
\end{align*}
Moreover, we drop the superscript $(i)$ so that $\widehat{c}^{\, n+1}_{h}$ refers to a specific component of the vector $\widehat{\bB{c}}^{n+1}_{h}$ (analogously for ${c}^{\, n}_{h}$). 
We have for all $L\in \mathcal{T}_h$
\begin{align*}
	- (\widehat{c}^{\, n+1}_L)_{\oplus}\, \psi^{\uparrow}_L \le 0\qquad \text{and}\qquad
	(\widehat{c}^{\, n+1}_{L})_{\oplus}\, \psi^{\downarrow}_{L} \le (\widehat{c}^{\, n+1}_{K_{\circ}})_{\oplus}\, \psi^{\downarrow}_{L}.
\end{align*}
For all $e\in \mathcal{E}_h^{\rm int}$, we use the short-hand notation $\bar{\beta}^e = \avg{\bB{\beta}}\cdot \bB{n}_e$ with $\bB{\beta} = -\nabla\mu_h^{n+1}$. 
Using the identity \eqref{eq:prop:o} and the above inequalities, we have
\begin{align*}
	&\adiffc\big(\widehat{c}^{\,n+1}_{h},\,\Pi_{\ominus}^0(\widehat{c}_h^{\,n+1}); \bB{\beta},\psi_h\big) \\
	&\quad = \sum_{\atopp{e\in\mathcal{E}_h^{\rm int}}{e=K_{\circ}\cap L}}
	\int_e \Big\{
	\bar{\beta}^e_{\oplus}\big((\widehat{c}^{\, n+1}_{K_{\circ}})_{\oplus}\psi^{\uparrow}_{K_{\circ}} + (\widehat{c}^{\, n+1}_{L})_{\oplus}\psi^{\downarrow}_{L}\big)
	\,-\,\bar{\beta}^e_{\ominus}\big((\widehat{c}^{\, n+1}_{L})_{\oplus}\psi^{\uparrow}_{L} + (\widehat{c}^{\, n+1}_{K_{\circ}})_{\oplus} \psi^{\downarrow}_{K_{\circ}} \big)\Big\}(\widehat{c}^{\, n+1}_{K_{\circ}})_{\ominus}\\
	&\quad = \sum_{\atopp{e\in\mathcal{E}_h^{\rm int}}{e=K_{\circ}\cap L}}
	\int_e \Big\{
	\bar{\beta}^e_{\oplus}(\widehat{c}^{\, n+1}_{L})_{\oplus}\psi^{\downarrow}_{L}
	\,- \,\bar{\beta}^e_{\ominus} (\widehat{c}^{\, n+1}_{L})_{\oplus} \psi^{\uparrow}_{L} \Big\}(\widehat{c}^{\, n+1}_{K_{\circ}})_{\ominus}
	\, \le\,  \sum_{\atopp{e\in\mathcal{E}_h^{\rm int}}{e=K_{\circ}\cap L}}
	\int_e \Big\{
	\bar{\beta}^e_{\oplus}(\widehat{c}^{\, n+1}_{L})_{\oplus}\psi^{\downarrow}_{L}\Big\}(\widehat{c}^{\, n+1}_{K_{\circ}})_{\ominus}\\[1ex]
	&\quad \le \sum_{\atopp{e\in\mathcal{E}_h^{\rm int}}{e=K_{\circ}\cap L}}
	\int_e \Big\{
	\bar{\beta}^e_{\oplus}(\widehat{c}^{\, n+1}_{K_{\circ}})_{\oplus}\psi^{\downarrow}_{L}
	\Big\}(\widehat{c}^{\, n+1}_{K_{\circ}})_{\ominus} = 0.
\end{align*}
Setting as test function $\varphi_h = \Pi_{\ominus}^0(\widehat{c}_h^{\,n+1})$ in \eqref{eq:fulld:c:hat}, using the obtained inequality for $\adiffc$, together with Lemma~\ref{lem:auxiliar} and \eqref{eq:ineq:o:1}, we can show the following:
\begin{align*}
	0 \, \le  \, \big((\widehat{c}_{h}^{n+1} - c_{h}^{n}),\Pi_{\ominus}^0(\widehat{c}_h^{n+1})\big)_{\Ltwo} \,=\,\, |K_{\circ}|(\widehat{c}_{K_{\circ}}^{n+1} - c_{K_{\circ}}^{n})(\widehat{c}_{K_{\circ}}^{n+1})_\ominus \,\le\, 0,
\end{align*}
from which we conclude that either $\widehat{c}_{K_{\circ}}^{n+1}={c}_{K_{\circ}}^{n}$ or $(\widehat{c}_{K_{\circ}}^{n+1})_\ominus = 0$ and therefore the lower bound $\bB{0}\leq \widehat{\bB{c}}_h^{n+1}$ holds. 
To prove the upper bound, we begin by defining the following discrete function
\begin{align*}
	\theta_h^{n+1} \coloneqq \sum_{i=1}^{\nsp} \widehat{c}^{\, (i),n+1}_h \in \mathcal{P}_0(\Omega).
\end{align*}
Using the positivity of $\widehat{u}^{\, n+1}_h$ from Lemma~\ref{lem:maxprinciple:u}, we have that $\widehat{u}_L^{\, n+1} = (\widehat{u}_L^{\, n+1})_{\oplus}$ and also $M^{\uparrow}(\widehat{u}_h^{\, n+1}) = \psi^{\uparrow}_h u_h^{n+1}$ and $ M^{\downarrow}(u_h^{n+1}) = \psi^{\downarrow}_hu_h^{n+1}$. 
Summing up the equations in \eqref{eq:fulld:c:hat} and subtracting the result from Equation~\eqref{eq:fulld:u:hat}, we get
\begin{align*}
	0 &= \dfrac{1}{\Delta t}\big(\widehat{u}_h^{\, n+1}-\theta_h^{n+1},\varphi_h\big)_{\Ltwo} +  \aconv\big(\widehat{u}_h^{\, n+1} - \theta_h^{n+1}; \bB{q}_h^{n}, \varphi_h\big) \,&\, \\[1ex]
	& \quad +   \sum_{e\in\mathcal{E}_h^{\rm int}}
	\int_e
	\bar{\beta}^e_{\oplus}\Big((\widehat{u}_K^{n+1}-\theta_K^{n+1})\psi^{\uparrow}_{K} + (\widehat{u}_L^{\, n+1}-\theta_L^{n+1})\psi^{\downarrow}_{L}\Big)\jump{\varphi_h}\\[-1.5ex]
	&\quad - \sum_{e\in\mathcal{E}_h^{\rm int}} \int_e
	\bar{\beta}^e_{\ominus}\Big((\widehat{u}_L^{\, n+1}-\theta_L^{n+1})\psi^{\uparrow}_{L} +  (\widehat{u}_K^{n+1}-\theta_K^{n+1})\psi^{\downarrow}_{K}\Big)\jump{\varphi_h}
\end{align*}
for all $\varphi_h\in \mathcal{P}_0(\Omega)$, which is equivalent to
\begin{align*}
	0 &  = \dfrac{1}{\Delta t}\big(\widehat{u}_h^{\, n+1}-\theta_h^{n+1},\varphi_h\big)_{\Ltwo} + \aconv\big(\widehat{u}_h^{\,n+1} - \theta_h^{n+1}; \bB{q}_h^{n}, \varphi_h\big)
	+ \adiffc\Big( \widehat{u}^{\, n+1}_h-\theta_h^{n+1}, \varphi_h;  - \nabla\mu_h^{n+1}, \widehat{u}_h^{\,n+1}\Big).
\end{align*}
Thanks to Lemma~\ref{lem:aux:existence}, the unique solution to the above equation is the null solution, that is $\widehat{u}_h^{\, n+1}-\theta_h^{n+1} \equiv 0$.
Therefore, we have shown that
\begin{align*}
	\sum_{i=1}^{\nsp} \widehat{c}^{\, (i),n+1}_h =  \widehat{u}_h^{\, n+1}\quad\text{and}\quad \widehat{\bB{c}}^{n+1}_h\leq\rho_{\rm b}\bB{1},
\end{align*}
with which we conclude the proof.
\end{proof}

Because of the definition of the numerical flux in $\adiffs$ (cf.~\eqref{def:op:c:s}), which does not tend to zero when the concentration of auxiliary substrates $\widehat{\bB{s}}_h^{n+1}$ reaches $\rhoL$, we limit ourselves to show only non-negativity of $\widehat{\bB{s}}_h^{n+1}$.
Furthermore, the dominating component of the substrate vector is water, so all other solutes concentrations are much smaller than~$\rhoL$, and water is not produced at any reaction, so its concentration cannot exceed~$\rhoL$.

\begin{lemma}\label{lem:maxprinciple:s}
Given the solution $\bB{q}^n_h\in \spaceq$ to~\eqref{syst:fulld:stokes}, $\mu_{h}^{n+1}\in \mathcal{P}^{\rm cont}_{1}(\Omega)$,  $\bB{c}^{n}_h\in [\mathcal{P}_0(\Omega)]^{\nsp}$ and $\widehat{u}_h^{\, n+1}\in \mathcal{P}_0(\Omega)$ such that $\bB{0}\le \bB{s}^n_h$ and $0\le \widehat{u}^{\, n+1}_h\le \rho_{\rm b}$ in $\overline\Omega$, then $\widehat{\bB{s}}^{n+1}_h \in \mathcal{P}_0(\Omega)$ computed from Equation~\eqref{eq:fulld:s:hat} satisfies
\begin{align*}
	\widehat{\bB{s}}^{n+1}_h \geq 0.
\end{align*}
\end{lemma}
\begin{proof}
The proof follows analogously to that of Lemma~\ref{lem:maxprinciple:c} as far as non-negativity is concerned. 
For this reason, we limit ourselves only to showing that the bilinear form $\adiffs$ is non-positive for the specific test function $\varphi_h = \Pi_{\ominus}^0(\widehat{s}_h^{\,n+1})$. We let $\widehat{s}_h^{n+1}$ be a single component of the vector $\widehat{\bB{s}}_h^{n+1}$.
Then, using the short-hand notation $\bB{\beta} = \nabla\mu_h^{n+1}$ and $\bar{\beta}^e = \avg{\bB{\beta}}\cdot \bB{n}_e$ for all $e\in \mathcal{E}_h^{\rm int}$, and the identity \eqref{eq:prop:o}, we have
\begin{align*}
	\adiffs\Big(&\widehat{s}^{\,n+1}_{h},\,\Pi_{\ominus}^0(\widehat{s}_h^{\,n+1}); \bB{\beta},\widehat{u}_h^{n+1} \Big)\\
	& = \sum_{\atopp{e\in\mathcal{E}_h^{\rm int}}{e=K_{\circ}\cap L}}
	\int_e \Big\{
	\bar{\beta}^e_{\oplus}M_{\bB{s}}\big(\widehat{u}^{n+1}_{K_\circ}\big)\big(\widehat{s}^{\, n+1}_{K_\circ}\big)_{\oplus}
	\,-\,\bar{\beta}^e_{\ominus}  M_{\bB{s}}\big(\widehat{u}^{n+1}_{L}\big)\big(\widehat{s}^{\, n+1}_{L}\big)_{\oplus} \Big\}\big(\widehat{s}^{\, n+1}_{K_{\circ}}\big)_{\ominus}\\
	& = -\sum_{\atopp{e\in\mathcal{E}_h^{\rm int}}{e=K_{\circ}\cap L}}
	\int_e
	\bar{\beta}^e_{\oplus}M_{\bB{s}}\big(\widehat{u}^{n+1}_{L}\big)\big(\widehat{s}^{\, n+1}_{L}\big)_{\oplus} \big(\widehat{s}^{\, n+1}_{K_{\circ}}\big)_{\ominus} \leq 0,
\end{align*}
which concludes the proof.
\end{proof}

To achieve the desired boundedness of the unknowns $\bB{c}_h^{n+1}$, $\bB{s}_h^{n+1}$ and $u_h^{n+1}$ computed from Equations~\eqref{eq:fulld:c}--\eqref{eq:fulld:u}, we follow the idea in~\cite{Burger2022a}.
We will bound the reaction terms making distinction between the positive and negative components of the stoichiometric matrices $\bB{\sigma}_{\bB{c}}$ and $\bB{\sigma}_{\bB{s}}$.
We require the following time-step restriction
\begin{align}\label{eq:CFL}
	\Delta t \,\leq\, \min\Big\{ \big(\max\big\{\Lambda_{\bB{c}},\Lambda_{\bB{s}}\big\}\boundr\big)^{-1},\varepsilon\big(\nsp\boundr_{\max}\big)^{-1}\Big\},
\end{align}
where the constants are introduced in Assumption~\ref{asumption:reactionterms} and
\begin{align*}
	\Lambda_{\bB{\xi}} := \max_{k} \sum_{j\in \mathcal{J}_{\bB{\xi},k}^{-}} \big(-\sigma_{\bB{\xi}}^{(k,j)}\big)\geq 0\qquad \text{for }\bB{\xi}\in \{\bB{c},\bB{s}\}.
\end{align*}
Finally, the main invariant-region principle for the discretized concentrations $u_h^{n+1}$, $\bB{c}_h^{n+1}$ and $\bB{s}_h^{n+1}$ is the following.

\begin{theorem} \label{thm:maxprinciple:ODEs}
Given the solution $\bB{q}^n_h\in \spaceq$ to \eqref{syst:fulld:stokes}, $\mu_{h}^{n+1}\in \mathcal{P}^{\rm cont}_{1}(\Omega)$,  $\bB{c}^{n}_h\in [\mathcal{P}_0(\Omega)]^{\nsp}$ and $\bB{s}^{n}_h\in [\mathcal{P}_0(\Omega)]^{\ssp}$, such that  $\bB{c}^{n}_h\geq \bB{0}$, $\bB{s}^{n}_h\geq \bB{0}$, and $u^{n}_h\geq 0$, Assumption~\ref{asumption:reactionterms} and the CFL condition \eqref{eq:CFL}.
Then the discrete concentrations $\bB{c}^{n+1}_h$, $\bB{s}^{n+1}_h$, and $u_h^{n+1}$ computed from \eqref{eq:fulld:c}, \eqref{eq:fulld:s}, and \eqref{eq:fulld:u}, respectively, fulfil:
\begin{align*}
	\bB{0}\leq \bB{c}_h^{n+1}\leq \rhoC\bB{1},\qquad \bB{s}_h^{n+1}\geq \bB{0},\qquad  0 \,\leq\,  u_h^{n+1}\, = \, \sum_{i=1}^{\nsp} c^{\, (i),n+1}_h \, \leq \,\rhoC.
\end{align*}
\end{theorem}
\begin{proof}
Given non-negative $\bB{c}^{n}_h$, $\bB{s}^{n}_h$, and $u^{n}_h$, Lemmas~\ref{lem:maxprinciple:u}, \ref{lem:maxprinciple:c} and \ref{lem:maxprinciple:s} ensure that $\widehat{\bB{c}}_h^{n+1}$, $\widehat{\bB{s}}_h^{n+1}$ and $\widehat{u}_h^{n+1}$ are non-negative.

We observe that Equations~\eqref{eq:fulld:c}--\eqref{eq:fulld:u} are in fact forward Euler approximations of an autonomous ODE system for each $K\in \mathcal{T}_h$, whose right hand side is given by the respective reaction terms, and the initial condition is $(\widehat{\bB{c}}_K^{n+1},\widehat{\bB{s}}_K^{n+1},\widehat{u}_K^{n+1})$.
Then, we rewrite \eqref{eq:fulld:c}--\eqref{eq:fulld:s} as follows
\begin{align*}
	\begin{pmatrix}
		\bB{c}_K^{n+1}\\
		\bB{s}_K^{n+1}
	\end{pmatrix}
	=
	\begin{pmatrix}
		\widehat{\bB{c}}_K^{n}\\
		\widehat{\bB{s}}_K^{n}
	\end{pmatrix}
	+ \Delta t
	\begin{pmatrix}
		\bB{\sigma}_{\bB{c}}\\
		\bB{\sigma}_{\bB{s}}
	\end{pmatrix} \bB{r}(\widehat{\bB{c}}_K^n,\widehat{\bB{s}}_K^n)
	,\qquad \forall K\in \mathcal{T}_h.
\end{align*}
Now, setting $\xi$ (resp.\ $\widehat{\xi}$) as the $k$-th component of the extended vector $(\bB{c}_K^{n+1},\bB{s}_K^{n+1})$  (resp.\ $(\widehat{\bB{c}}_K^{n+1},\widehat{\bB{s}}_K^{n+1})$) and let $\bB{\xi} =\bB{c}$ if $k\leq \nsp$, and $\bB{\xi}=\bB{s}$ otherwise, we have
\begin{align*}
	\xi^{n+1}_K & = \widehat{\xi}_K^{n+1} +
	\Delta t \sum_{j\in \mathcal{J}_{\bB{\xi},k}^{-}} \sigma_{\bB{\xi}}^{(k,j)} \bar{r}^{(j)}(\widehat{\bB{c}}_K^{n+1},\widehat{\bB{s}}_K^{n+1})\widehat{\xi}^{n+1}_K +
	\Delta t \sum_{j\in \mathcal{J}_{\bB{\xi},k}^{+}} r^{(j)}(\widehat{\bB{c}}_K^{n+1},\widehat{\bB{s}}_K^{n+1})\\
	& \geq \widehat{\xi}_K^{n+1} +
	\Delta t \sum_{j\in \mathcal{J}_{\bB{\xi},k}^{-}} \sigma_{\bB{\xi}}^{(k,j)} \bar{r}^{(j)}(\widehat{\bB{c}}_K^{n+1},\widehat{\bB{s}}_K^{n+1})\widehat{\xi}^{n+1}_K\\
	&  = \big(1 -
	\Delta t \Lambda_{\bB{\xi}}\boundr\big)\widehat{\xi}^{n+1}_K,
\end{align*}
where we have used the boundedness of $\bar{r}^{(j)}$ for all $j\in \mathcal{J}_{\bB{\xi},k}^{-}$ in Assumption~\ref{asumption:reactionterms}.
Hence, provided that $\widehat{\xi}_K^{n+1}\geq 0$ and under condition \eqref{eq:CFL}, we can conclude that $\bB{c}^{n+1}_K\geq \bB{0}$ and $\bB{s}^{n+1}_K\geq \bB{0}$ for all $K\in \mathcal{T}_h$. 
Furthermore, adding up \eqref{eq:fulld:c} for $1\leq i\leq \nsp$ and thanks to Lemma~\ref{lem:maxprinciple:c}, we conclude that
\begin{align*}
	u_h^{n+1} =  \sum_{i=1}^{\nsp} c^{\, (i),n+1}_h \geq 0.
\end{align*}

To show the upper bound of $u_h^{n+1}$, we follow the same idea as in \cite[Lemma~3]{Burger2022a}. 
If $u_K^{n+1}\geq \rhoC - \varepsilon$, thanks to Assumption~\ref{asumption:reactionterms}, we have that $\widetilde{\mathcal{R}}(\widehat{\bB{c}}_K^{n+1},\widehat{\bB{s}}_K^{n+1}) = 0$ and therefore $u_K^{n+1} = \widehat{u}_K^{n+1}\leq \rhoC$.
Otherwise, when $u_K^{n+1}\leq \rhoC-\varepsilon$, then
\begin{align*}
	u_K^{n+1} & < \rhoC-\varepsilon + \Delta t\, \widetilde{\mathcal{R}}\big(\widehat{\bB{c}}_K^{n+1},\widehat{\bB{s}}_K^{n+1}\big)%
	\leq \rhoC - \varepsilon + \Delta t\, \nsp\boundr_{\max}\leq \rhoC,
\end{align*}
for all $K\in \mathcal{T}_h$ with $\Delta t \leq \varepsilon / (\nsp\boundr_{\max})$, and in consequence $\bB{c}_h^{n+1}\leq \rhoC\bB{1}$.
\end{proof}

In Theorem~\ref{thm:maxprinciple:ODEs}, we explicitly show that the sum of all components of $\bB{c}_h^{n+1}$ equals $u_h^{n+1}$ since this equality was not strongly imposed in the system of equations \eqref{syst:scheme:fulld}.
To save memory consumption, one component of the vector $\bB{c}_h^{n+1}$ can be computed from the others, reducing the system by one equation. The next two propositions are straightforward results showing the conservation of mass property of the scheme \eqref{syst:scheme:fulld} and the boundedness of the auxiliary unknown $\widetilde{u}_h^{n+1}\in \mathcal{P}_1^{\rm cont}(\Omega)$ under mass lumping.

\begin{proposition}
The auxiliary discrete concentrations $\widehat{u}_h^{n+1}$, $\widehat{\bB{c}}_h^{n+1}$ and $\widehat{\bB{s}}_h^{n+1}$, solutions to Equations  \eqref{eq:fulld:u:hat}, \eqref{eq:fulld:c:hat}, \eqref{eq:fulld:s:hat} and \eqref{eq:other:u}, respectively, fulfil
\begin{align}\label{eq:conservation:ucs}
	\int_{\Omega} \widehat{u}_h^{n+1} = \int_{\Omega} u_h^n,\qquad  \int_{\Omega} \widehat{\bB{c}}_h^{n+1} = \int_{\Omega} \bB{c}_h^n,\qquad  \int_{\Omega} \widehat{\bB{s}}_h^{n+1} = \int_{\Omega} \bB{s}_h^n\,,\quad\text{and}\quad \int_{\Omega} \widetilde{u}_h^{n+1} = \int_{\Omega} u_h^n.
\end{align}
In addition, the total concentration $u_h^{n+1} + s_{{\rm tot},h}^{n+1}$, where $s_{{\rm tot},h}^{n+1}$ is the sum of the components of $\bB{s}_h^{n+1}$, preserves the mass.
\end{proposition}
\begin{proof}
Using the test function $\varphi_h = 1\in\mathcal{P}_0(\Omega)$ in \eqref{eq:fulld:u:hat}--\eqref{eq:fulld:s:hat} and $\vartheta_h=1\in\mathcal{P}_1^{\rm cont}(\Omega)$ in \eqref{eq:other:u}, we get Equations~\eqref{eq:conservation:ucs}.
Then, adding up \eqref{eq:fulld:s:hat} by components together with \eqref{eq:fulld:u}, using Assumption~\ref{asumption:reactionterms} and integrating the result over $\Omega$, we obtain the desired conservation of mass.
\end{proof}

\begin{proposition}\label{prop:mass-lumping}
If mass lumping is used to compute the auxiliary variable $\widetilde{u}_h^{n+1}$ in \eqref{eq:other:u}, then $0\leq \widetilde{u}_h^{n+1}\leq \rhoC$ in $\overline{\Omega}$ for $n\geq 0$.
\end{proposition}
\begin{proof}
See~Appendix~\ref{appendix:A}.
\end{proof}

For the Stokes system \eqref{syst:fulld:stokes}, we mention that the pair of discrete spaces $\spaceq$ and $\spacep$ is one of the stable pairs for which discrete inf-sup conditions can be shown \cite{Boffi2013} with the additional property of being the simplest pair of mass-preserving spaces.
The existence and uniqueness of this decoupled system are presented next.

\begin{lemma} \label{lem:solv:stokes}
Given $u_h^{n+1}\in\mathcal{P}_0(\Omega)$ and $\widetilde{u}_h^{n+1}\in \mathcal{P}^{\rm cont}_{1}(\Omega)$, then there exists a unique solution $(\bB{q}_h^{n+1},p_h^{n+1})\in \spaceq \times \spacep$  to the Stokes system~\eqref{syst:fulld:stokes} with forcing term given by $\mathbf{f}^{n+1}_h$ (cf.~Equation~\ref{eq:discrete:source:f}).
Moreover, there exists a constant $C>0$, independent of $h$, such that the following a-priori estimate holds
\begin{align*}
	\|\bB{q}_h^{n+1}\|_{1,\Omega} + \|p_h^{n+1}\|_{0,\Omega}\leq C \Big( g\Delta\rho|\Omega|^{1/2} + \Psi_{\max}' \|\nabla \widetilde{u}_h^{n+1}\|_{0,\Omega}\Big).
\end{align*}
\end{lemma}
\begin{proof}
We first introduce the discrete bilinear forms $\bB{a}_h:\spaceq\times\spaceq\to \mathbb{R}$ and $\bB{b}_h:\spaceq\times \spacep\to \mathbb{R}$ defined by
\begin{align*}
	\bB{a}_h(u_h^{n+1}; \bB{q}_h,\bB{w}_h):=\big( \nu(u_h^{n+1})\bB{e}(\bB{q}_h), \bB{e}(\bB{w}_h) \big)_{\Ltwo},\qquad
	\bB{b}_h(\varphi_h, \bB{w}_h):= \big( \varphi_h, {\rm div}(\bB{w}_h) \big)_{\Ltwo}.
\end{align*}
Then, we can write the Stokes system \eqref{syst:fulld:stokes} as the following conforming saddle point problem: Find $(\bB{q}_h^{n+1},p_h^{n+1})\in\spaceq\times \spacep$ such that
\begin{alignat*}{3}
	&\bB{a}_h(u_h^{n+1}; \bB{q}_h^{n+1},\bB{w}_h) & \,-\, \bB{b}_h(p_h^{n+1}, \bB{w}_h) &= -(\mathbf{f}^{n+1}_h,\bB{w}_h)_{\Ltwo}\,\qquad && \forall \bB{w}_h\in \spaceq,\\
	& & \bB{b}_h(\varphi_h, \bB{q}_h^{n+1})  & = 0\,\qquad && \forall \varphi_h \in\spacep.
\end{alignat*}
Next, the ellipticity of $\bB{a}_h$ onto $\spaceq\subseteq \mathbf{H}_0^1(\Omega)$ is concluded thanks to Korn and Poincaré inequalities in the following way
\begin{align*}
	\bB{a}_h(u_h^{n+1};\bB{w}_h,\bB{w}_h) & \,\geq \, \nu_{\rm min} \big(\bB{e}(\bB{w}_h), \bB{e}(\bB{w}_h) \big)_{\Ltwo}
	\,\geq\, \dfrac{\nu_{\rm min}}{2} \|\nabla\bB{w}_h\big\|_{0,\Omega}^2
	\,\geq\, \alpha_{\rm ell}\|\bB{w}_h\big\|_{1,\Omega}^2,
\end{align*}
for all $\bB{w}_h\in \spaceq$, where $\alpha_{\rm ell}:= (\nu_{\rm min} c_{\rm p})/2>0$ is the ellipticity constant with $c_{\rm p}>0$ being the constant of Poincaré's inequality, and $\nu_{\rm \min}\coloneqq\min\{\nu_{\rm b},\nu_{\rm f}\}$.
On the other hand, using \cite[Proposition 8.4.3.]{Boffi2013}, we can conclude that the bilinear form $\bB{b}_h$ defined over the pair of spaces $\spaceq$ and $\spacep$ fulfils the inf-sup condition \cite[Eq.~(8.2.16)]{Boffi2013}, and therefore the chosen pair of spaces is stable. 
Making use of the upper bound of $u_h^{n+1}$ in Lemma~\ref{lem:maxprinciple:u} and inequality~\eqref{eq:bound:dPsi}, we see that the source term is bounded by
\begin{align*}
	\|\mathbf{f}_h^{n+1}\|_{0,\Omega}\leq \dfrac{g\Delta \rho}{\rhoC}\|u_h^{n+1}\|_{0,\Omega} + \|\eta\Psi'(u_h^{n+1})\nabla \widetilde{u}_{h}^{n+1}\|_{0,\Omega}\leq g\Delta \rho|\Omega|^{1/2} + \eta\Psi_{\max}' \|\nabla \widetilde{u}_{h}^{n+1}\|_{0,\Omega},
\end{align*}
where we have used the fact that $\div(\nabla \widetilde{u}_h^{n+1}) = 0$.
Finally, we can conclude the claim by simply applying \cite[Proposition 2.42]{Ern2004}.
\end{proof}

\section{Numerical simulations}\label{sec:simulations}
For the numerical simulations presented in this section, we have implemented the fully coupled numerical scheme~\eqref{syst:scheme:fulld} and \eqref{syst:fulld:stokes} in Python making use of the open source finite element library FEniCS (legacy version 2019.1.0)~\cite{FEniCS}.
For all examples, the nonlinear system in system~\eqref{syst:scheme:fulld} is solved by the Newton-Raphson algorithm with null initial guess.
We use an absolute and relative tolerance of $10^{-8}$, and the solution of tangent systems resulting from the linearization is performed making use of the multi-frontal massively parallel sparse direct solver MUMPS~\cite{Amestoy2000}.

In order to show the performance of the method with the boundary conditions assumed in the analysis in Section~\ref{sec:scheme:properties}, we will perform two simulations with a simple domain and initial conditions.
In addition, we will demonstrate that the model and numerical scheme can be extended to the case of other boundary conditions by simulating the upper part of an SSF with in- and outflows.

\subsection{Simplified reaction model}
We consider a reduced system with $\kC = \kL = 2$ components representing a biofilm, composed of one heterotrophic solid phase with concentration~$c^{(1)}$ and one phototrophic $c^{(2)}$, with dissolved oxygen, acting as a nutrient, of concentration~$s^{(1)}$ in water $s^{(2)}$.
Since we have
\begin{equation*}
	u = c^{(1)} + c^{(2)},\qquad \rhoL\left(1 - \frac{u}{\rhoC}\right) = s^{(1)} + s^{(2)},
\end{equation*}
we can obtain $c^{(2)}$ and $s^{(2)}$ by solving only Equation~\eqref{eq:solids} for~$c^{(1)}$ and~\eqref{eq:substrates} for~$s^{(1)}$.
For simplicity, we will consider only two reactions: heterotrophs feeding on phototrophs at the reaction rate~$r^{(1)}$ and phototrophs growing off nutrients at the rate~$r^{(2)}$, where the latter uses $c^{(1)}$ as a proxy for nutrients obtained from heterotroph death.
In the context of~Section~\ref{sec:model:reactions}, the reaction terms for the solid and substrate components are defined by the following stoichiometric matrices and vector of reaction rates:
\begin{align*}
	\bB{\sigma}_{\bB{c}} = \begin{bmatrix}
		1 & -1/2\\
		-1/2 & 1
	\end{bmatrix}, \qquad
	\bB{\sigma}_{\bB{s}} = -\dfrac{1}{2}\begin{bmatrix}
		1 & 1
	\end{bmatrix},\qquad
	\bB{r}(\bB{c},\bB{s}) = c^{(1)} c^{(2)} s^{(1)}\begin{bmatrix}
		\displaystyle R_1 \frac{1}{s^{(1)} + \mathcal{K}_{11}}\frac{1}{c^{(2)} + \mathcal{K}_{12}}
		\\[1.5ex]
		\displaystyle R_2  \frac{1}{s^{(1)} + \mathcal{K}_{21}}\frac{1}{c^{(1)} + \mathcal{K}_{22}}
	\end{bmatrix},
\end{align*}
where~$R_1$ and~$R_2$ [\unit{\per\ourtime}] are maximal rate constants and $\mathcal{K}_{ij}$ half-saturation constants.

\subsection{Interaction of two circular masses}
For this example, we set the unit square domain $\Omega = [0, 1]^2$ and the boundary conditions used in the analysis of the method, which are zero flux for the concentrations and no-slip conditions for the velocity field (cf.~\eqref{def:boundary:conditions:mu}--\eqref{def:boundary:conditions:mu}). Save for $R_1$ and $R_2$, the model parameters are
\begin{align} \label{eq:parameters:sim}
\begin{split}
	\rhoC = \qty{1117}{\concentration}, \quad 
	\rhoL = \qty{998}{\concentration}, \quad 
	\lambda = \qty{200}{\mobility}, \quad
	\kappa = \qty{0.5e-8}{\kappaunit},\quad 
	\phi_{\ast} = 0.01, \\
	\gamma = 0,\quad 
	\eta = \qty{1e-4}{\capillary},\quad 
	\nu_{\rm b} = \qty{1e-3}{\viscosity},\quad 
	\nu_{\rm f} = \qty{1}{\viscosity},\quad 
	g = \qty{9.81}{\acceleration},\\
	\mathcal{K}_{11} = \mathcal{K}_{21} = \qty{2e-2}{\concentration},\quad  
	\mathcal{K}_{12} = \qty{1e-2}{\concentration},
	\quad \mathcal{K}_{22} = \qty{4e-2}{\concentration}
	.
\end{split}
\end{align}
These parameters are chosen to show the qualitative behaviour without too long simulations; therefore, the time snapshots in figures are fractions of a second.
To define the initial conditions of each concentration variable, we consider smooth decaying functions ${C_0}$ with a specified constant value $\phi_0$ inside a circular disc defined by
\begin{equation*}
	{C_0}(\bm{x}; \bm{x}_0, r_0, \phi_0) \coloneqq \displaystyle \phi_0\frac{1}{2}\left( \tanh\left( \frac{r_0 - \| \bm{x} - \bm{x}_0 \|}{\sqrt{2\kappa}}\right) + 1\right),
\end{equation*}
for all $\bB{x} = (x_1,x_2)\in \Omega$, where $\bB{x}_0\in\Omega$ is the centre of a circle with $r_0>0$ its radius.

The initial state of the solid sub-phases consists of two masses in circular discs, where the top one has a considerably higher biomass concentration, with a mass of substrate $s^{({1})}$ at the centre of the domain.
More precisely, we have
\begin{align} \label{eq:initial:sim}
\begin{aligned}
	u(\bm{x}, 0) &= \rhoC{C_0}\big(\bm{x}; (\widetilde{x}_{\rm L}, 0.75)^\ttt, 0.2, 0.02\big)
	+ \rhoC{C_0}\big(\bm{x}; (\widetilde{x}_{\rm R}, 0.35)^\ttt, 0.2, 0.008\big),\\[1ex]
	c^{(1)}(\bm{x}, 0) &= \rhoC{C_0}\big(\bm{x}; (\widetilde{x}_{\rm L}, 0.75)^\ttt, 0.2, 0.01\big)
	+ \rhoC{C_0}\big(\bm{x}; (\widetilde{x}_{\rm R}, 0.35)^\ttt, 0.2, 0.002\big), \\[1ex]
	s^{(1)}(\bm{x}, 0) &= \rhoL{C_0}\big(\bm{x}; (0.5, 0.5)^\ttt, 0.25, 0.2\big).
	\end{aligned}
\end{align}
where $\widetilde{x}_{\rm L}$ and $\widetilde{x}_{\rm R}$ respectively assume the two values~$\widetilde{x}_{\rm L}=0.35$ and $\widetilde{x}_{\rm R} = 0.65$ (Figure~\ref{fig:blob0}), and~$\widetilde{x}_{\rm L}=0.45$ and $\widetilde{x}_{\rm R} = 0.55$ (Figure~\ref{fig:blob:reactions}), where the latter means that the two masses are closer initially.

\begin{figure}[!t]
\centering
\includegraphics[scale=1]{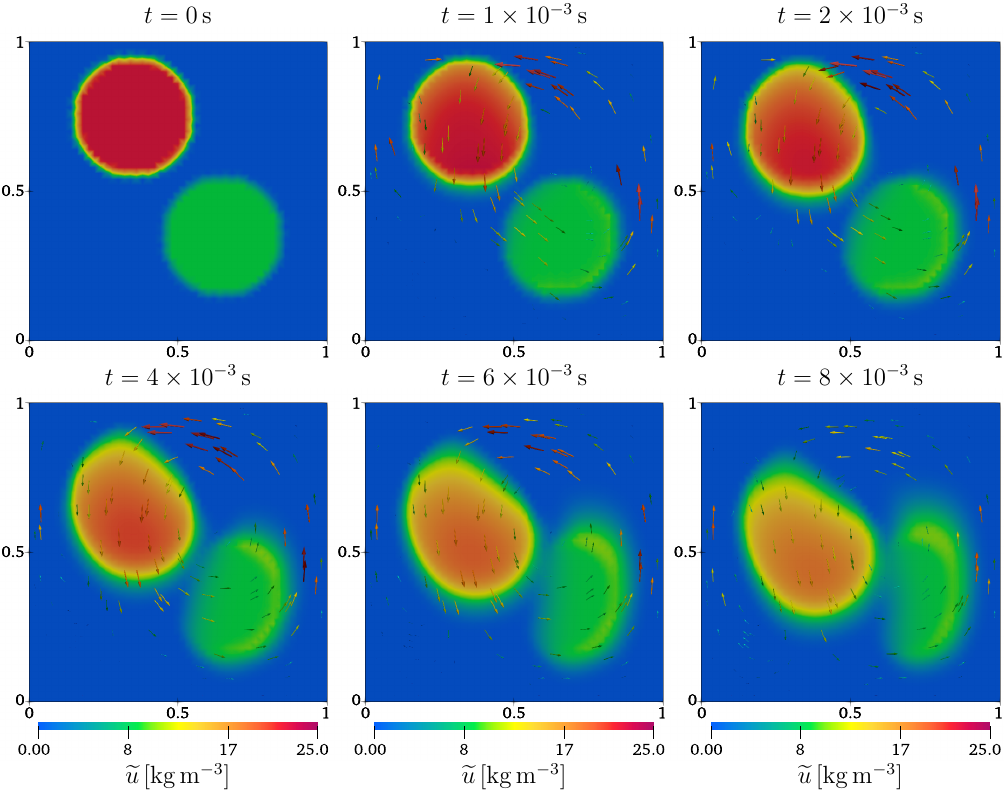}
	\caption{Interaction of two circular masses a distance apart without any reactions: Biofilm concentration $\widetilde{u}$ and mixture velocity $\bB{q}$ (arrows) with zero boundary conditions at six time points, with $R_1=R_2=0$. Both initial circles have the same radius $0.2$, and are initially centred at $(0.35, 0.75)^\ttt$ and $(0.65, 0.35)^\ttt$, respectively.}
	\label{fig:blob0}
\end{figure}%

\begin{figure}[!t]
\centering
\includegraphics[scale=1]{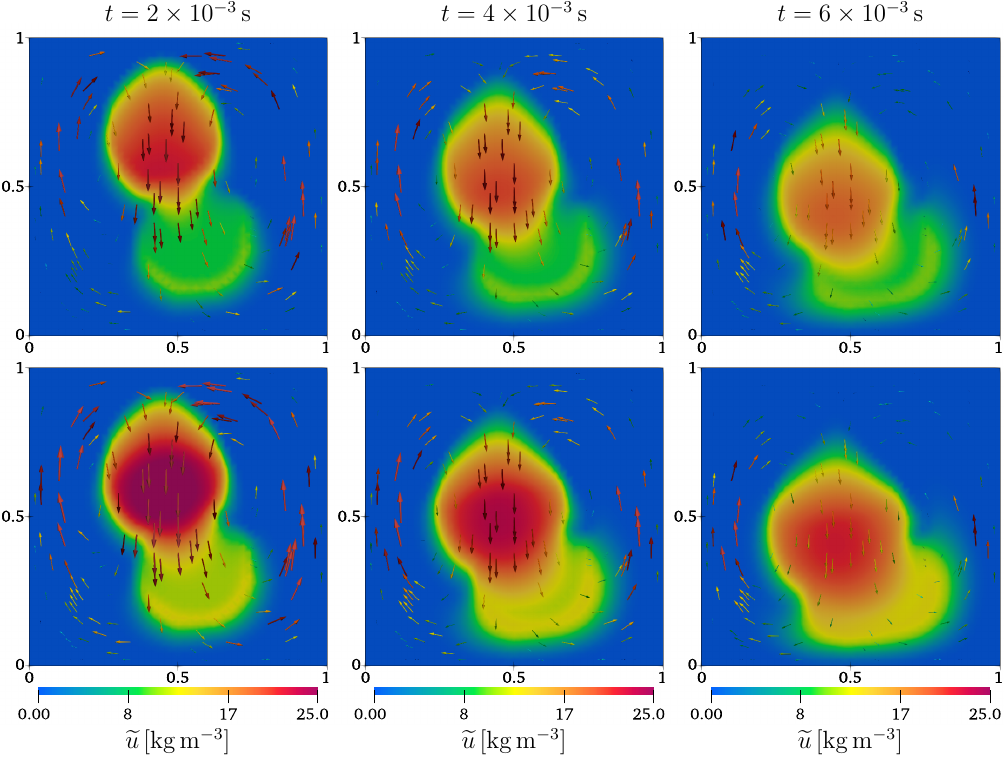}
	\caption{Interaction of two circular masses closer together: Biofilm concentration $\widetilde{u}$ and mixture velocity $\bB{q}$ (arrows) with zero boundary conditions for the case without reactions (top) and with reactions $R_1 = 100$~\unit{\per\ourtime} and $R_2 = 1000$~\unit{\per\ourtime} (bottom) at three time points.
	Both the initial discs of masses have the same radius~$0.2$, and are centred at $(0.45, 0.75)^\ttt$ and $(0.55, 0.35)^\ttt$, respectively.
	}
	\label{fig:blob:reactions}
\end{figure}%

\begin{figure}[!t]
	\centering
	\includegraphics[width=0.95\textwidth]{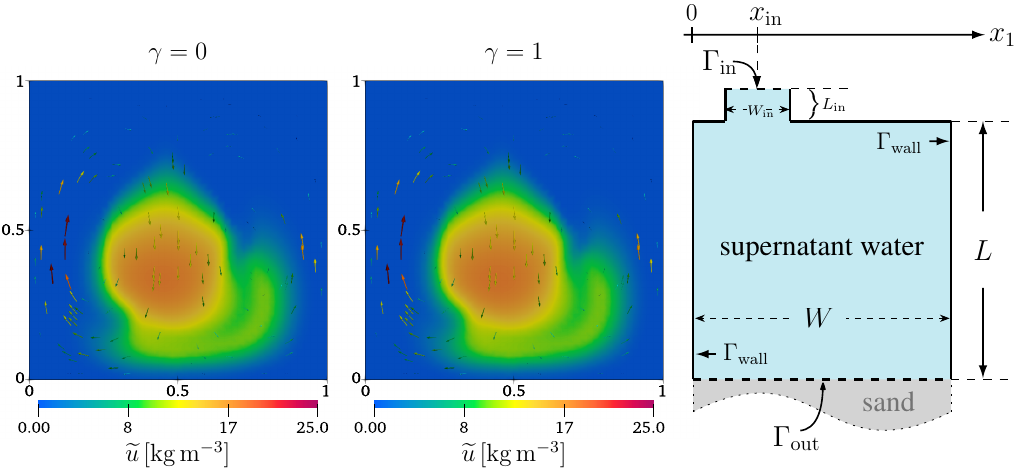}
	\caption{Interaction of two circular masses: Comparison of the approximated $\widetilde{u}$ with $\gamma=0$ (left plot) and $\gamma=1$ (middle plot) at $t= 8\times10^{-3}\,\rm s$. Right plot: Schematic of the supernatant water region of an SSF with inflow at the top boundary in $\Gamma_{\rm in}$ and outflow at the bottom boundary $\Gamma_{\rm out}$. On $\Gamma_{\rm wall}$, no-slip boundary conditions are considered.
	\label{fig:blob_gamma:ssf_schematic}}
\end{figure}

\subsubsection*{The effect of reactions}

For the first example, we consider the initial conditions in \eqref{eq:initial:sim} with the two circular regions defining~$u$ and~$c^{(1)}$ with the centre of the upper circle at $\widetilde{x}_L = 0.35$, and centre of the lower circle at $\widetilde{x}_{\rm R} = 0.65$. We set the reaction rates to $R_1 =R_2= 0$. In Figure~\ref{fig:blob0}, we show snapshots of the approximate concentration~$\widetilde{u}$ computed at six time points $t= k\times10^{-3}\,\rm s$, with $k=0,1,2,4,5,8$.
The arrows in the figure represent the approximate vector field $\bB{q}$ where the colours are according to the magnitude $\|\bB{q}\|_{0,\Omega}$. As expected, the heavier top mass settles faster than the lighter bottom one. The latter is pushed into a thin layer before the former catches up with it. This movement induces a circular flow, and the discs are deformed due to the mass transfer and the Cahn-Hilliard based relative velocity \eqref{eq:vrel}. Although the deformations of both circular mass regions are prominent, the masses have not completely merged at $t=8\times10^{-3}\,\rm s$.

Next, we consider initial circular-disc masses for~$u$ and~$c^{(1)}$ in a closer configuration with $\widetilde{x}_{\rm L} = 0.45$ and $\widetilde{x}_{\rm R} = 0.55$ in \eqref{eq:initial:sim}. In addition, we study the effect of the reaction terms by comparing two cases, one without reactions ($R_1 = R_2 = 0$) and one with non-zero reaction terms ($R_1 = 100$~\unit{\per\ourtime}, $R_2 = 1000$~\unit{\per\ourtime}). Figure~\ref{fig:blob:reactions} shows snapshots of $\widetilde{u}$ in the two cases.
When no reactions are present, the evolution of the system is mainly driven by gravity and convection and diffusive fluxes. However, when reactions are present, the mass transfer between species makes a significant change and the two discs of concentration merge smoothly already at $t=6\times10^{-3}\,\rm s$. With reactions, we observe higher values for the biofilm concentration, with the fused mass more spread out than in the case without reactions. Furthermore, unlike the simulation shown in Figure~\ref{fig:blob0}, the smaller distance between the initial discs gives rise to a less pronounced bulk velocity which also facilitates the blending. The individual concentrations~$c^{(1)}$ and $c^{(2)}$ in these simulations, follow roughly the same shape as the entire biofilm phase concentration~$u$, wherefore we do not plot them.

\subsubsection*{Mobility with $\gamma=0$ and $\gamma=1$}

To visualize the possible influence the parameter~$\gamma$ at the mobility function \eqref{eq:M} has on the output, we consider the scenario with reactions shown in Figure~\ref{fig:blob:reactions}, and set $\gamma=1$. Such a mobility function can be found in the model presented in~\cite{Chatelain2011b}. We show a comparison between the cases $\gamma = 0$ and $\gamma = 1$ at time $t=8\times10^{-3}\,\rm s$ in the first two plots in Figure~\ref{fig:blob_gamma:ssf_schematic}.
There is no noticeable difference, which can be explained by the fact that the values of~$u$ are much smaller than~$\rhoC$, making the term $(1 - u/\rhoC)^{1 + \gamma}$ play a much smaller role as $u/\rhoC$ tends to be small independent of the value of~$\gamma$.

\subsection{Simulation of  biofilm growth in the supernatant water of an SSF}

In the upcoming examples, we consider an inflow boundary condition at a piece of the boundary~$\Gamma_{\rm in}\subset\Gamma$ and an outflow boundary condition at $\Gamma_{\rm out}\subset \Gamma$, extending the boundary conditions given in \eqref{def:boundary:conditions:q}--\eqref{def:boundary:conditions:mu}. We consider the following two-dimensional domain of the supernatant water region where the origin is located at the lower left corner; see Figure~\ref{fig:blob_gamma:ssf_schematic} (right).
Let $\Omega_{\rm filter} \coloneqq [0, W] \times [0, L]$ be the longitudinal cut of the supernatant water region above the sand bed of an SSF; thus, with width~$W$ and height~$L$.
We let $\Omega_{\rm inlet}$ be a rectangular inlet centred at~$x_{\rm in}$ with width~$W_{\rm in}$ and pipe length~$L_{\rm in}$ given by
\begin{align*}
\Omega_{\rm inlet} \coloneqq \Big[x_{\rm in} - \tfrac{W_{\rm in}}{2}, x_{\rm in} + \tfrac{W_{\rm in}}{2}\Big] \times \Big[L, L+L_{\rm in}\Big].
\end{align*}
The entire domain in this case is $\Omega = \Omega_{\rm filter} \cup \Omega_{\rm inlet}$ and the inlet boundary $\Gamma_{\rm inlet}$ is located at the top segment of $\Omega_{\rm inlet}$, while the outlet boundary $\Gamma_{\rm out}$ is the bottom segment of $\Omega_{\rm filter}$. The remaining part of the boundary is denoted by $\Gamma_{\rm wall}$.
These segments are defined by
\begin{align*}
 \Gamma_{\rm in} \coloneqq \Big[x_{\rm in} - \tfrac{W_{\rm in}}{2}, x_{\rm in} + \tfrac{W_{\rm in}}{2}\Big] \times \big\{L+L_{\rm in}\big\},\qquad \Gamma_{\rm out} \coloneqq [0, W] \times \{0\},\qquad
 \Gamma_{\rm wall} \coloneqq \Gamma \setminus (\Gamma_{\rm in} \cup \Gamma_{\rm out}).
\end{align*}
To describe the in- and outflow boundary conditions related to $\bB{q}$, we set a maximum inlet velocity $q_{\max} \geq 0$, and define parabolic inlet and outlet profiles $q_{\rm in}$,  $q_{\rm out}$ as
\begin{align*}
	q_{\rm in}(x) &\coloneqq \frac{4q_{\max}}{W_{\rm in}^2}\Big(x - \big(x_{\rm in} - \tfrac{W_{\rm in}}{2}\big)\Big)\Big(x - \big(x_{\rm in} + \tfrac{W_{\rm in}}{2}\big)\Big)\quad \text{and}\quad
	q_{\rm out}(x) \coloneqq \frac{4q_{\max}}{W^2} \Big(\dfrac{W_{\rm in}}{W}\Big) x (x - W).
\end{align*}
We consider the following non-zero Dirichlet boundary conditions for the velocity
\begin{align*}
	\bm{q} &\coloneqq \begin{cases}
		\bm{0}, &\text{on $\Gamma_{\rm wall}$,}\\
		(0, q_{\rm in}(x_1)) & \text{on $\Gamma_{\rm in}$,} \\
		(0, q_{\rm out}(x_1)) & \text{on $\Gamma_{\rm out}$,} 
	\end{cases}
\end{align*}
and the following alternative flux conditions for the rest of the system ($i,j=1,2$)
\begin{alignat*}{4}
	(\bm{q} + M(u)\nabla\mu)\cdot\nn &= 0 \quad &&\text{on $\Gamma$,}\\
	c^{(i)}(\bm{q} + M_{\bB{c}}(u)\nabla\mu)\cdot\nn &= 0 \quad &&\text{on $\Gamma$,} \\
	s^{(j)}(\bm{q} + M_{\bB{c}}(u)\nabla\mu)\cdot\nn &= 0 \quad &&\text{on $\Gamma_{\rm wall}$,} \\
	s^{(j)} &= s^{(j)}_{\rm in}  \quad &&\text{on $\Gamma_{\rm in}$,}
\end{alignat*}
with constant incoming concentration $s_{\rm in}^{(j)} \geq 0$ and free outflow of $s^{(j)}$ through $\Gamma_{\rm out}$.

To define the initial concentrations, we define a vertical step function $\mathcal{H}_0$ given by a constant value $\phi_0$ below a certain height $y_0$ and zero otherwise:
\begin{equation*}
	\mathcal{H}_0(\bm{x}; y_0, \phi_0) \coloneqq \begin{cases}
		\phi_0, & x_2 < y_0, \\
		0, & x_2 \geq y_0.
	\end{cases}
\end{equation*}
Then, we set the initial state of the filter as
\begin{align*}
	u(\bB{x}, 0 ) & = \rhoC\mathcal{H}_0\big(\bm{x}; 0.6, 0.05\big), \qquad
	c^{(1)}(\bB{x}, 0 ) = \rhoC\mathcal{H}_0\big(\bm{x}; 0.6, 0.025\big), \qquad
	s^{(1)}(\bB{x}, 0 ) = 0.
\end{align*}
In order to highlight the qualitative behaviour of the biofilm model, we exaggerate physical quantities and use the following parameters for simulation:
\begin{gather*}
	\lambda = \qty{150}{\mobility}, \quad
	\kappa = \qty{1e-8}{\kappaunit},\quad 
	\eta = \qty{1e-6}{\capillary},\quad 
	W = \qty{0.5}{\metre},\quad
	L = \qty{0.5}{\metre}, \\
	x_{\rm in} = \qty{0.1}{\metre}, \quad
	W_{\rm in} = \qty{0.1}{\metre},\quad
	L_{\rm in} = \qty{0.025}{\metre}, \quad
	q_{\rm in} = \qty{300}{\metre\per\ourtime}, \quad
	s^{(1)}_{\rm in} = \num{0.99}\rhoL,
\end{gather*}
where the rest of parameters are taken as in \eqref{eq:parameters:sim}, while $R_1$ and $R_2$ are specified next.

\begin{figure}[!t]
\centering
\includegraphics[scale=1]{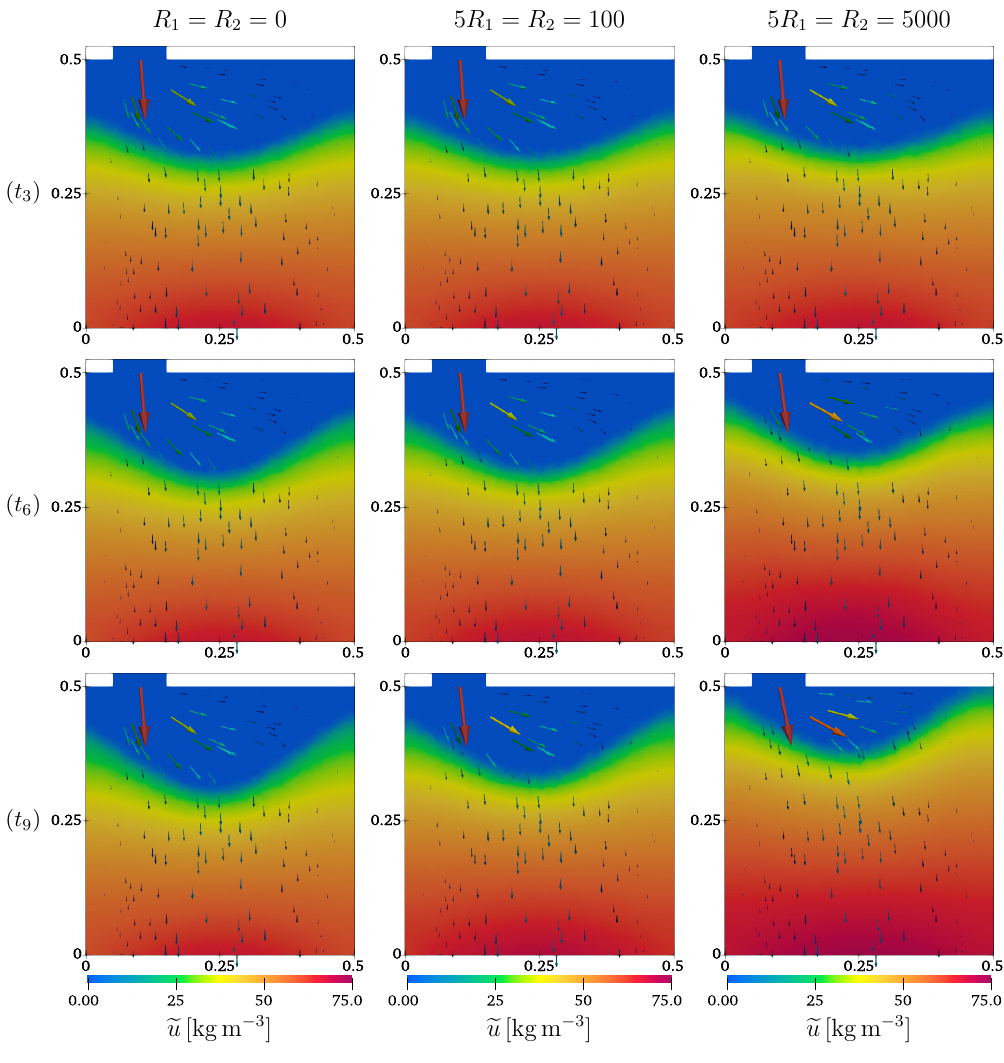}
	\caption{Simulations of an SSF: Total biofilm concentration $\widetilde{u}$ and mixture velocity $\bB{q}$ (arrows) with in- and outflow boundary conditions for zero, medium and large reactions (columns). The time points (rows) are $t_3 \coloneqq 3\times10^{-3}\,\rm s$, $t_6 \coloneqq 6\times10^{-3}\,\rm s$ and $t_9 \coloneqq 9\times10^{-3}\,\rm s$.}
	\label{fig:filter}
\end{figure}%

Simulations for the described scenario are shown in Figures \ref{fig:filter} and \ref{fig:filter_components}. In Figure~\ref{fig:filter}, we show the impact of increasing the reaction rates on the total biofilm concentration $\widetilde{u}$. To illustrate this, we simulate three cases. The first case without reaction terms with $R_1=R_2 = 0$ (1st column), the second with $R_1 = \qty{20}{\per\ourtime}$ and $R_2 = \qty{100}{\per\ourtime}$ (2nd column), and the third with $R_1 = \qty{1000}{\per\ourtime}$ and $R_2 = \qty{5000}{\per\ourtime}$ (3rd column). Three time snapshots (rows) are shown at $t_3 = 3\times10^{-3}\,\rm s$ (1st row), $t_6= 6\times 10^{-3}\,\rm s$ (2nd row) and $t_9 = 9\times10^{-3}\,\rm s$ (3rd row).
The first no-reaction case reveals that as the system is fed with substrate $s^{(1)}_{\rm in}$, the main variation of $\widetilde{u}$ is taking place in an interval of values near $25\,\rm kg\,m^{-3}$, which can be influenced by the dynamics of the velocity field $\bB{q}$ and diffusive terms. With higher reactions (2nd and 3rd columns), the simulations look qualitatively the same at~$t=t_3$. However, at $t_6$ and $t_9$, we can observe the actual biofilm growth, which due to the bulk flow~$\bB{q}$, form a convex shape. As expected, larger reactions imply a more prominent biofilm growth.

We now choose the medium case in Figure~\ref{fig:filter} and show in Figure~\ref{fig:filter_components} snapshots of the concentration of each species. At the four times shown (including $t_1 = 10^{-3}\,\rm s$), we observe that, while $c^{(1)}$ decreases in a region next to the middle line $x_1 = W/2$, the concentration of $c^{(2)}$ gradually increases and form an oval region of high concentration towards $\Gamma_{\rm out}$. On the other hand, since $s^{(1)}$ is fed at $\Gamma_{\rm in}$, this concentration increases and moves downwards, while $s^{(2)}$ is consumed. The black curve in Figure~\ref{fig:filter_components} is the solid-liquid interface defined by a large gradient of~$\widetilde{u}$. In all simulations, the concentrations stay positive and bounded.

\begin{figure}[!t]
\centering
\includegraphics[scale=1]{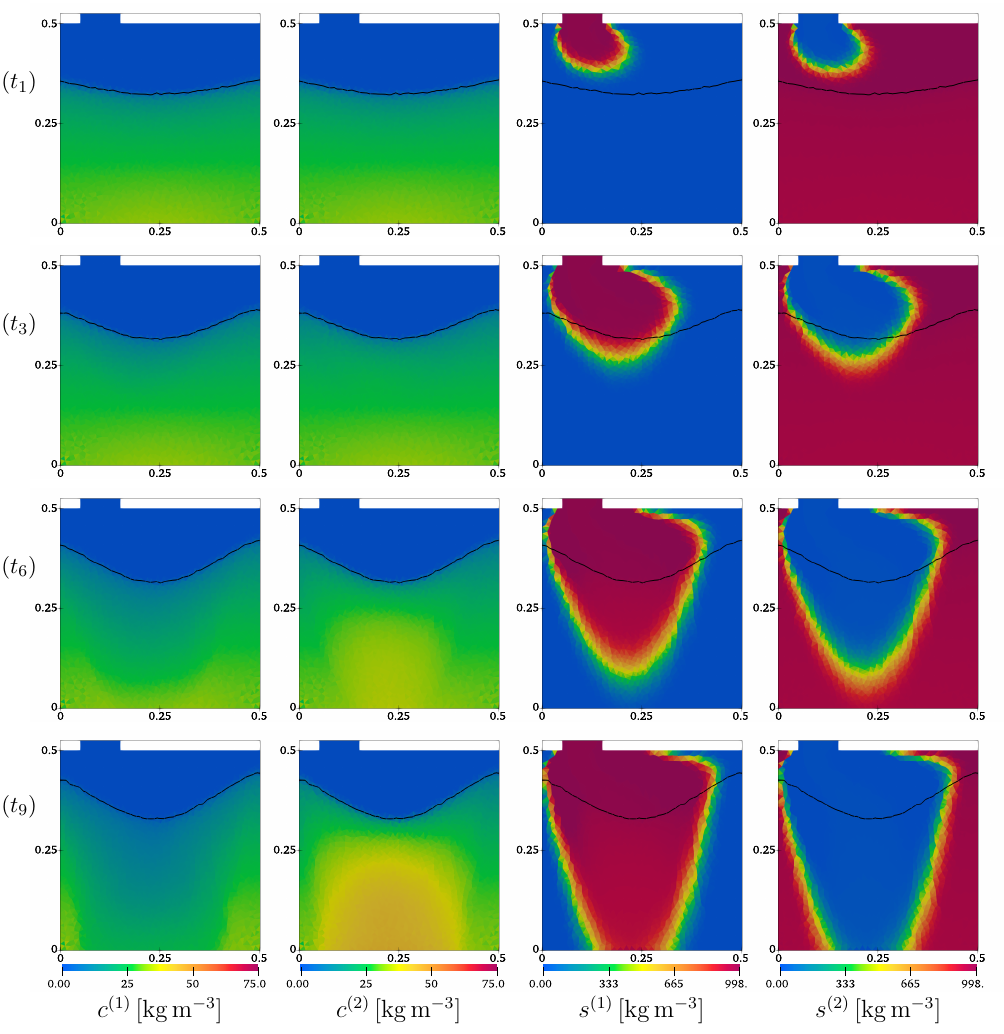}
\caption{Simulation of an SSF: Component concentrations $c^{(1)}$ (1st column), $c^{(2)}$ (2nd column), $s^{(1)}$ (3rd column) and $s^{(2)}$ (4th column) for the case of in- and outflow boundary conditions and $5R_1=R_2 = \qty{100}{\per\ourtime}$. The time points are $t_1 = 10^{-3}\,\rm s$, $t_3 = 3\times10^{-3}\,\rm s$, $t_6 = 6\times10^{-3}\,\rm s$ and $t_9 = 9\times10^{-3}\,\rm s$. The black contour line corresponds to the solid-liquid interface of $\widetilde{u}$.}
	\label{fig:filter_components}
\end{figure}%

\section{Conclusions}\label{sec:conclusions}

We introduce a model of biofilm growth in several spatial dimensions within the upper part of an SSF that incorporates the transport of solid components and soluble substrates, and stationary Stokes equations for the mixture flow.
At the top of the sand bed in an SSF, biofilm establishes, the \emph{Schmutzdecke}, and grows upwards into the supernatant water of the SSF due to biological processes.
The total biofilm concentration is modelled by an advective Cahn--Hilliard equation with reactions, which connects all the equations of the coupled system, and therefore plays a central role in the dynamics of the process. 

We exploit the assumption that the relative velocity between the solid and liquid phases depends on the total biofilm concentration and formulate the system of equations following the structure arising in models of reactive sedimentation~\cite{Burger2021,Burger2022a}.
The nonlinear reaction terms in the vector equations for the solid and substrate subphases drive mass transfer between such, particularly, the consumption of soluble substrates.
Some key ingredients in the derivation of the approximate momentum equation are the smallness of the density difference between the solid and liquid phases, and the fact that in SSF when no in- or outflows take place, the acceleration terms can be neglected, thus arriving at the Stokes equations.
Although we consider a stationary flow equation, the coupling with the concentration equations induce implicitly a time dependency on the velocity field.
We also include a viscosity function, which depends on the total biofilm concentration as a convex combination of the two main phases and a source term accounting for capillary forces~\cite{Dlotko2022, Han2021}.

For the numerical treatment, we extend the discontinuous-Galerkin-based numerical scheme proposed in \cite{Acosta2023a} to include the coupling with vector equations of convection-reaction type. The scheme is designed to incorporate upwind numerical fluxes to approximate the integrals over the edges of the triangulation arising from the integration-by-parts process. Both the discrete formulation and flux approximation of the solid components' equations are developed such that the sum of the solid components is equal to the total biofilm concentration at each time point. In addition, the time splitting of the reaction terms makes it possible to prove the non-negativity of the unknown concentrations.
For the flow equations, we employ a conservative primal formulation for the velocity-pressure pair, where the velocity field is in an $H^1$-conforming subspace while the pressure is piecewise constant. The main result of the paper is an invariant-region property of the discrete concentration unknowns, namely that all concentrations are non-negative and the total solids concentration has the solid density as the upper bound; hence, all solid subphases are also bounded above.

Numerical simulations of a simplified model of biofilm growth with two solids and two liquid phases demonstrate the performance of the model, and the invariant-region property of the numerical scheme.
Two examples show cases with and without reactions.
Another example shows that the scheme in fact can be straightforwardly adapted to handling in- and outflow boundary conditions.
That example is in qualitative agreement with the one-dimensional simulations shown in~\cite{Diehl2025}, where a more comprehensive biological model is used.
In principle, any biological model could be used within the presented model.

Future work can include different directions.
The SSF model in~\cite{Diehl2025} includes partly the processes of biofilm growth and flows within the pores of the entire sand bed below the supernatant water., and partly another main phase consisting of suspended particles in the liquid phase.
This suspension can be entrapped within the biofilm in addition to flowing outside the biofilm, and transfer terms in the form of discrete diffusion between these two volumes are present in addition to biological reactions.
Still other transfer terms are the attachment of solid particles to sand grains and the biofilm, and detachment of solids from these as an increasing function of the mixture velocity.
In two or three dimensions, this could be modelled by considering an interface problem between supernatant water and sand, where the flow in the sand can be modelled by Darcy's law.

As for the numerical scheme of the concentration variables, different numerical fluxes can be analysed and higher-order schemes can be achieved, for example, by using higher-order polynomial reconstructions. For the approximation of the flow equations, several options have already been mentioned in the introduction, from which we observe the compatibility of using a mass-conservative scheme such as the one in~\cite{Oyarzua2023}.
Formulations in which the velocity is field is neither $H^1$-conforming nor $H(\div)$-conforming need to be further investigated.

\section*{Acknowledgments}
JC is supported by ANID through Fondecyt project 3230553. SD and JM acknowledge support from the Swedish Research Council (Vetenskapsr{\aa}det 2019-04601) and the Swedish Research Council for Sustainable Development (FORMAS 2022-01900). The computations were enabled by resources provided by the National Academic Infrastructure for Supercomputing in Sweden (NAISS), partially funded by the Swedish Research Council through grant agreement no. 2022-06725.

\appendix
\section{Postponed proofs.}\label{appendix:A}

\begin{proof}[Proof of Lemma~\ref{lem:auxiliar}]
We first observe that since $\div(\bB{\beta}) = 0$ holds on $\Omega$, the vector function $\bB{\beta}\in \mathbf{H}(\div;\Omega)$ (or $\mathbf{H}^1(\Omega)$) is continuous across each edge $e\in\mathcal{E}_h^{\rm int}$, and the normal component $\beta^e := \bB{\beta}\cdot \bB{n}_e$ is well defined. Then, using the divergence theorem on each element $K\in \mathcal{T}_h$, we obtain
\begin{align*}
	\sum_{{e\in\mathcal{E}_h^{\rm int}}\atop{e\subseteq \partial K}}\int_e \beta^e=0,\qquad \forall K\in \mathcal{T}_h.
\end{align*}
Then, using the above result combined with the fact that $\varphi_L \ge \varphi_{K_{\circ}}$ for all $L\in \mathcal{T}_h$, and observing that $ \Pi_{\ominus}^{\alpha}(\varphi)$ is zero everywhere but in $K_{\circ}$, we can show the first inequality in \eqref{ineq:aux:aconv} as follows
\begin{align*}
	\aconv\big(\varphi, \Pi_{\ominus}^{\alpha}(\varphi); \bB{\beta}\big)
	& = \sum_{{e\in\mathcal{E}_h^{\rm int}}\atop{e=K_{\circ}\cap L}}\int_e\Big(
	\beta^e_{\oplus}\varphi_{K_{\circ}} - \beta^e_{\ominus}\varphi_{L}\Big)(\varphi_{K_{\circ}} - \alpha)_{\ominus}\\
	& \le \sum_{{e\in\mathcal{E}_h^{\rm int}}\atop{e=K_{\circ}\cap L}}\int_e\Big(\beta^e_{\oplus}\varphi_{K_{\circ}} - \beta^e_{\ominus}\varphi_{K_{\circ}}\Big)
	(\varphi_{K_{\circ}} - \alpha)_{\ominus}
	 = \varphi_{K_{\circ}}(\varphi_{K_{\circ}}-\alpha)_{\ominus}\sum_{{e\in\mathcal{E}_h^{\rm int}}\atop{e \subseteq\partial K_{\circ}}}\int_e \beta^e=0.
\end{align*}
For the second inequality in \eqref{ineq:aux:aconv}, we have that $\varphi_L \le \varphi_{K^{\diamond}}$ for all $L\in \mathcal{T}_h$, and proceeding similarly as before, we obtain
\begin{align*}
	\aconv\big(\varphi, \Pi_{\oplus}^{\alpha}(\varphi); \bB{\beta}\big)
	& \ge \varphi_{K^{\diamond}}(\varphi_{K^{\diamond}} - \alpha)_{\oplus}\sum_{{e\in\mathcal{E}_h^{\rm int}}\atop{e = K^{\diamond}\cap L}}\int_e\beta^e=0.
\end{align*}
Thus, we conclude the proof.
\end{proof}

\begin{proof}[Proof of Lemma~\ref{prop:mass-lumping}]
We begin by rewriting Equation~\eqref{eq:other:u} as the following linear system
\begin{align*}
\mathbb{M}[\widetilde{u}_h^{n+1}] = [\zeta],
\end{align*}
where $[\widetilde{u}_h^{n+1}]$ denotes the vector of degrees of freedom (d.o.f.) of $\widetilde{u}_h^{n+1}\in\mathcal{P}_1^{\rm cont}(\Omega)$, and each $i$-th component of the vector $[\zeta]$ is given by $[\zeta]_i = \big(u^{n+1}_h, \vartheta^{(i)}_h\big)_{0,\Omega}$, where $\vartheta^{(i)}_h$ is the Lagrange basis function of the continuous space $\mathcal{P}_1^{\rm cont}(\Omega)$, associated to the $i$-th vertex of the triangulation. We approximate $\mathbb{M}$ by a diagonal matrix $\widetilde{\mathbb{M}}$ with positive entries via row-summation, i.e. $\widetilde{\mathbb{M}}_{ii} = \sum_{j} \mathbb{M}_{ij}$. Then, inverting this matrix is straightforward, and we can readily obtain
\begin{equation*}
	[\widetilde{u}_h^{n+1}]_i = \widetilde{\mathbb{M}}_{ii}^{-1}{[b]_i}.
\end{equation*}
Noting that $\sum_{j} \vartheta^{(j)}_h \equiv 1$ in $\Omega$, and since $u_h^{n+1}$ is piecewise constant, we have that
\begin{align*}
	\widetilde{\mathbb{M}}_{ii} &= \sum_{j} \int_\Omega  \vartheta^{(i)}_h \vartheta^{(j)}_h = \int_\Omega \vartheta^{(i)}_h > 0 \qquad\text{and}\qquad 	[\zeta]_i = \big(u^{n+1}_h, \vartheta^{(i)}_h\big)_{0,\Omega} = \sum_{K \in \mathcal{T}_h}u^{n+1}_{h,K} \int_K \vartheta^{(i)}_h > 0,
\end{align*}
from which it easily follows that $[\widetilde{u}_h^{n+1}]_i > 0$ and, using the fact that $u^{n+1}_{h,K}\leq \rhoC$, we get
\begin{align*}
	[\widetilde{u}_h^{n+1}]_i &= \dfrac{1}{\widetilde{\mathbb{M}}_{ii}}\sum_{K \in \mathcal{T}_h}u^{n+1}_{h,K} \int_K \vartheta^{(i)}_h
	\leq \dfrac{\rhoC}{\widetilde{\mathbb{M}}_{ii}} \sum_{K \in \mathcal{T}_h}\int_K \vartheta^{(i)}_h = \rhoC,
\end{align*}
with which we conclude that $0\leq \widetilde{u}_h^{n+1}\leq \rhoC$.
\end{proof}

\bibliographystyle{plain}


\begin{thebibliography}{10}

\bibitem{Abels2013}
H.~Abels, D.~Depner, and H.~Garcke.
\newblock On an incompressible {N}avier--{S}tokes/{C}ahn--{H}illiard system
  with degenerate mobility.
\newblock {\em Ann. Inst. Henri Poincare (C) Anal. Non Lineaire},
  30(6):1175--1190, 2013.

\bibitem{Acosta2023a}
D.~Acosta-Soba, F.~Guillén-González, and J.R. Rodr\'iguez-Galv\'an.
\newblock A structure-preserving upwind dg scheme for a degenerate phase-field
  tumor model.
\newblock {\em Comput. Math. Appl.}, 152:317--333, 2023.

\bibitem{Acosta2023b}
D.~Acosta-Soba, F.~Guillén-González, and J.R. Rodr\'iguez-Galv\'an.
\newblock An upwind {DG} scheme preserving the maximum principle for the
  convective {C}ahn-{H}illiard model.
\newblock {\em Numer. Algorithms}, 92:1589--1619, 2023.

\bibitem{Acosta2025}
D.~Acosta-Soba, F.~Guillén-González, J.R. Rodr\'iguez-Galv\'an, and J.~Wang.
\newblock Property-preserving numerical approximation of a
  {C}ahn-{H}illiard-{N}avier-{S}tokes model with variable density and
  degenerate mobility.
\newblock {\em Appl. Numer. Math.}, 209:68--83, 2025.

\bibitem{FEniCS}
M.S. Aln{\ae}s, J.~Blechta, J.~Hake, A.~Johansson, B.~Kehlet, A.~Logg,
  C.~Richardson, J.~Ring, M.E. Rognes, and G.N. Wells.
\newblock The {FE}ni{CS} project version 1.5.
\newblock {\em Arch.~Numer.~Softw.}, 3(100):9--23, 2015.

\bibitem{Amestoy2000}
P.R. Amestoy, I.S. Duff, and J.-Y. L'Excellent.
\newblock Multifrontal parallel distributed symmetric and unsymmetric solvers.
\newblock {\em Comput. Methods Appl. Mech. Engrg.}, 184(2–4):501--520, 2000.

\bibitem{Zou2022}
G.~an~Zou, B.~Wang, and X.~Yang.
\newblock A fully-decoupled discontinuous {G}alerkin approximation of the
  {C}ahn--{H}illiard--{B}rinkman--{O}hta--{K}awasaki tumor growth model.
\newblock {\em M2AN Math. Model. Numer. Anal.}, 56(6):2141--2180, 2022.

\bibitem{Banas2008}
L'. Ba$\rm\check{n}$as and R.~N\"{u}rnberg.
\newblock Adaptive finite element methods for {C}ahn--{H}illiard equations.
\newblock {\em J. Comput. Appl. Math.}, 218(1):2--11, 2008.

\bibitem{Barrett1999}
J.W. Barrett, J.F. Blowey, and H.~Garcke.
\newblock Finite element approximation of the {C}ahn-{H}illiard equation with
  degenerate mobility.
\newblock {\em SIAM J. Num. Anal.}, 37(1):286--318, 2000.

\bibitem{Boffi2013}
D.~Boffi, F.~Brezzi, and M.~Fortin.
\newblock {\em Mixed Finite Element Methods and Applications}.
\newblock Springer Berlin Heidelberg, 2013.

\bibitem{Burger2021}
R.~B\"urger, J.~Careaga, and S.~Diehl.
\newblock A method-of-lines formulation for a model of reactive settling in
  tanks with varying cross-sectional area.
\newblock {\em IMA J. Appl. Math.}, 86(3):514--546, 2021.

\bibitem{Burger2022a}
R.~B\"{u}rger, J.~Careaga, S.~Diehl, and R.~Pineda.
\newblock A moving-boundary model of reactive settling in wastewater treatment.
  {P}art~2: {N}umerical scheme.
\newblock {\em Appl. Math. Modelling}, 111:247--269, 2022.

\bibitem{Burger2023b}
R.~B\"{u}rger, J.~Careaga, S.~Diehl, and R.~Pineda.
\newblock Numerical schemes for a moving-boundary convection-diffusion-reaction
  model of sequencing batch reactors.
\newblock {\em ESAIM: Math.~Model.~Numer.~Anal.}, 57(5):2931--2976, 2023.

\bibitem{Cahn1959}
J.W. Cahn.
\newblock Free energy of a nonuniform system. {II.} {T}hermodynamic basis.
\newblock {\em J. Chem. Phys.}, 30(5):1121--1124, 1959.

\bibitem{Cahn1958}
J.W. Cahn and J.E. Hilliard.
\newblock Free energy of a nonuniform system. {I}. {I}nterfacial free energy.
\newblock {\em J. Chem. Phys.}, 28(2):258--267, 1958.

\bibitem{Cahn1959b}
J.W. Cahn and J.E. Hilliard.
\newblock Free energy of a nonuniform system. {III.} {N}ucleation in a
  two-component incompressible fluid.
\newblock {\em J. Chem. Phys.}, 31(3):688--699, 1959.

\bibitem{Oyarzua2023}
J.~Cama{\~{n}}o and R.~Oyarz\'ua.
\newblock A conforming and mass conservative pseudostress-based mixed finite
  element method for {S}tokes.
\newblock {\em Preprint 2023-15, Centro de Investigaci\'on en Ingenier\'ia
  Matem\'atica (CI{$^2$}MA), Universidad de Concepci\'on}, pages 1--24, 2023.

\bibitem{Careaga2024}
J.~Careaga and V.~Osores.
\newblock A multilayer shallow water model for polydisperse reactive
  sedimentation.
\newblock {\em Appl. Math. Model.}, 134:570--590, 2024.

\bibitem{Chatelain2011a}
C.~Chatelain, T.~Balois, P.~Ciarletta, and M.~Ben Amar.
\newblock Emergence of microstructural patterns in skin cancer: a phase
  separation analysis in a binary mixture.
\newblock {\em New~J.~Phys.}, 13(11):115013, 2011.


\bibitem{Chatelain2011b}
C.~Chatelain, P.~Ciarletta, and M.~Ben Amar.
\newblock Morphological changes in early melanoma development: Influence of
  nutrients, growth inhibitors and cell-adhesion mechanisms.
\newblock {\em J.~Theor.~Biol.}, 290:46–59, 2011.

\bibitem{Cherfils2014}
L.~Cherfils, A.~Miranville, and S.~Zelik.
\newblock On a generalized {C}ahn-{H}illiard equation with biological
  applications.
\newblock {\em Discrete Contin. Dyn. Syst. Ser. B}, 19(7):2013--2026, 2014.

\bibitem{Clavijo2019}
S.P. Clavijo, A.F. Sarmiento, L.F.R. Espath, L.~Dalcin, A.M.A. Cortes, and V.M.
  Calo.
\newblock Reactive $n$-species {C}ahn--{H}illiard system: {A}
  thermodynamically-consistent model for reversible chemical reactions.
\newblock {\em J. Comput. Appl. Math.}, 350:143--154, 2019.

\bibitem{Diehl2025}
S.~Diehl, J.~Manríquez, C.J. Paul, and T.~Rosenqvist.
\newblock A convection-diffusion-reaction system with discontinuous flux
  modelling biofilm growth in slow sand filters.
\newblock {\em Appl. Math. Model.}, 137:115675, 2025.

\bibitem{Dlotko2022}
T.~Dlotko.
\newblock {N}avier–{S}tokes–{C}ahn–{H}illiard system of equations.
\newblock {\em J.~Math.~Phys.}, 63(11), 2022.

\bibitem{Elliott1996}
C.M. Elliott and H.~Garcke.
\newblock On the {C}ahn-{H}illiard equation with degenerate mobility.
\newblock {\em SIAM J. Math. Anal.}, 27(2):404--423, 1996.

\bibitem{Elson2010}
E.L. Elson, E.~Fried, J.E. Dolbow, and G.M. Genin.
\newblock Phase separation in biological membranes: {I}ntegration of theory and
  experiment.
\newblock {\em Annu. Rev. Biophys.}, 39(1):207--226, 2010.

\bibitem{Ern2004}
A.~Ern and J.-L. Guermond.
\newblock {\em Theory and Practice of Finite Elements}.
\newblock Springer New York, 2004.

\bibitem{Eymard2000}
R.~Eymard, T.~Gallou\"et, and R.~Herbin.
\newblock {\em Finite volume methods}, pages 713--1018.
\newblock Elsevier, 2000.

\bibitem{Furihata2001}
D.~Furihata.
\newblock A stable and conservative finite difference scheme for the
  {C}ahn-{H}illiard equation.
\newblock {\em Numerische Mathematik}, 87(4):675--699, 2001.

\bibitem{Gatica2011}
G.N. Gatica, L.F. Gatica, and A.~M\'arquez.
\newblock Augmented mixed finite element methods for a vorticity‐based
  velocity--pressure--stress formulation of the {S}tokes problem in 2{D}.
\newblock {\em Int. J. Numer. Methods Fluids}, 67(4):450--477, 2011.

\bibitem{Gatica2010}
G.N. Gatica, A.~M\'arquez, and M.A. S\'anchez.
\newblock Analysis of a velocity--pressure--pseudostress formulation for the
  stationary {S}tokes equations.
\newblock {\em Comput. Methods Appl. Mech. Engrg.}, 199(17--20):1064--1079,
  2010.

\bibitem{Giorgini2020}
A.~Giorgini and R.~Temam.
\newblock Weak and strong solutions to the nonhomogeneous incompressible
  {N}avier-{S}tokes-{C}ahn-{H}illiard system.
\newblock {\em J. Math. Pures Appl.}, 144:194--249, 2020.

\bibitem{Guillen2014}
F.~Guill\'en-Gonz\'alez and G.~Tierra.
\newblock Splitting schemes for a {N}avier-{S}tokes-{C}ahn-{H}illiard model for
  two fluids with diﬀerent densities.
\newblock {\em J. Comput. Math.}, 32(6):643--664, 2014.

\bibitem{Guillen2024b}
F.~Guill\'en-Gonz\'alez and G.~Tierra.
\newblock Structure preserving finite element schemes for the
  {N}avier--{S}tokes--{C}ahn--{H}illiard system with degenerate mobility.
\newblock {\em Comput. Math. Appl.}, 172:181--201, 2024.

\bibitem{Guillen2024a}
F.~Guillén-González and G.~Tierra.
\newblock Energy-stable and boundedness preserving numerical schemes for the
  {C}ahn--{H}illiard equation with degenerate mobility.
\newblock {\em Appl. Numer. Math.}, 196:62--82, 2024.

\bibitem{Guo2022}
Z.~Guo, Q.~Cheng, P.~Lin, C.~Liu, and J.~Lowengrub.
\newblock Second order approximation for a quasi-incompressible
  {N}avier-{S}tokes {C}ahn-{H}illiard system of two-phase flows with variable
  density.
\newblock {\em J. Comput. Phys.}, 448:110727, 2022.

\bibitem{Han2021}
D.~Han, X.~He, Q.~Wang, and Y.~Wu.
\newblock Existence and weak--strong uniqueness of solutions to the
  {C}ahn-{H}illiard-{N}avier-{S}tokes-{D}arcy system in superposed free flow
  and porous media.
\newblock {\em Nonlinear Analysis}, 211:112411, 2021.

\bibitem{Han2015}
D.~Han and X.~Wang.
\newblock A second order in time, uniquely solvable, unconditionally stable
  numerical scheme for {C}ahn--{H}illiard--{N}avier--{S}tokes equation.
\newblock {\em J. Comput. Phys.}, 290:139--156, 2015.

\bibitem{Jingxue1992}
Y.~Jingxue.
\newblock On the existence of nonnegative continuous solutions of the
  {C}ahn-{H}illiard equation.
\newblock {\em J. Diff. Equations}, 97(2):310--327, 1992.

\bibitem{Kay2009}
D.~Kay, V.~Styles, and E.~S\"{u}li.
\newblock Discontinuous {G}alerkin finite element approximation of the
  {C}ahn-{H}illiard equation with convection.
\newblock {\em {SIAM} J. Num. Anal.}, 47(4):2660--2685, 2009.

\bibitem{Khain2008}
E.~Khain and L.M. Sander.
\newblock Generalized {C}ahn-{H}illiard equation for biological applications.
\newblock {\em Phys. Rev. E.}, 77(5), 2008.

\bibitem{Klapper2006}
I.~Klapper and J.~Dockery.
\newblock Role of cohesion in the material description of biofilms.
\newblock {\em Phys. Rev. E}, 74(3), 2006.

\bibitem{Li2013}
Y.~Li, D.~Jeong, J.~Shin, and J.~Kim.
\newblock A conservative numerical method for the {C}ahn--{H}illiard equation
  with {D}irichlet boundary conditions in complex domains.
\newblock {\em Comput. Math. Appl.}, 65(1):102--115, 2013.

\bibitem{Maiyo2023}
J.K. Maiyo, S.~Dasika, and C.T. Jafvert.
\newblock Slow sand filters for the 21st century: A review.
\newblock {\em Int. J. Environ. Res. Public Health}, 20(2):1019, Jan 2023.

\bibitem{Schijven2013}
J.F. Schijven, H.H.J.L. van~den Berg, M.~Colin, Y.~Dullemont, W.A.M. Hijnen,
  A.~Magic-Knezev, W.A. Oorthuizen, and G.~Wubbels.
\newblock A mathematical model for removal of human pathogenic viruses and
  bacteria by slow sand filtration under variable operational conditions.
\newblock {\em Water Res.}, 47(7):2592--2602, 2013.

\bibitem{Shin2011}
J.~Shin, D.~Jeong, and J.~Kim.
\newblock A conservative numerical method for the {C}ahn--{H}illiard equation
  in complex domains.
\newblock {\em J. Comput. Phys.}, 230(19):7441--7455, 2011.

\bibitem{Song2020}
S.~Song, L.~Rong, K.~Dong, X.~Liu, P.~Le Clech, and Y.~Shen.
\newblock Particle-scale modelling of fluid velocity distribution near the
  particles surface in sand filtration.
\newblock {\em Water Res.}, 177:115758, 2020.

\bibitem{Stoter2023}
S.K.F. Stoter, T.B. van Sluijs, T.H.B. Demont, E.H. van Brummelen, and C.V.
  Verhoosel.
\newblock Stabilized immersed isogeometric analysis for the
  {N}avier--{S}tokes--{C}ahn--{H}illiard equations, with applications to
  binary-fluid flow through porous media.
\newblock {\em Comput. Methods Appl. Mech. Engrg.}, 417:116483, 2023.

\bibitem{Eikelder2024}
M.F.P. ten Eikelder and D.~Schillinger.
\newblock The divergence-free velocity formulation of the consistent
  {N}avier-{S}tokes {C}ahn--{H}illiard model with non-matching densities,
  divergence-conforming discretization, and benchmarks.
\newblock {\em J. Comput. Phys.}, 513:113148, 2024.

\bibitem{Trikannad2023}
S.A. Trikannad, D.~van Halem, J.W. Foppen, and J.P. van~der Hoek.
\newblock The contribution of deeper layers in slow sand filters to pathogens
  removal.
\newblock {\em Water Res.}, 237:119994, 2023.

\bibitem{Wang2007}
Z.~Wang and T.~Hillen.
\newblock Classical solutions and pattern formation for a volume filling
  chemotaxis model.
\newblock {\em Chaos}, 17(3), 2007.

\bibitem{Wu2022}
H.~Wu.
\newblock A review on the {C}ahn-{H}illiard equation: classical results and
  recent advances in dynamic boundary conditions.
\newblock {\em Electron. Res. Arch.}, 30(8):2788--2832, 2022.

\bibitem{Xia2007}
Y.~Xia, Y.~Xu, and C.-W. Shu.
\newblock Local discontinuous {G}alerkin methods for the {C}ahn--{H}illiard
  type equations.
\newblock {\em J. Comput. Phys.}, 227(1):472--491, 2007.

\bibitem{Zhang2008}
T.~Zhang, N.G. Cogan, and Q.~Wang.
\newblock {P}hase field models for biofilms. {I}. {T}heory and one-dimensional
  simulations.
\newblock {\em SIAM J. Appl. Math.}, 69(3):641--669, 2008.

\end{thebibliography}

\end{document}